\documentclass[11pt,a4paper]{article}
\usepackage[latin1]{inputenc}
\usepackage[english]{babel}
\usepackage[T1]{fontenc}
\usepackage{amssymb}
\usepackage{amsmath,amsfonts}
\usepackage{graphicx}
\usepackage{indentfirst}
\usepackage{dsfont}
\usepackage{color}
\usepackage{enumitem}
\usepackage{float}
\usepackage{pict2e}
\usepackage{multicol}

\newtheorem{theorem}{Theorem}[section]
\newtheorem{lemma}[theorem]{Lemma}

\newtheorem{corollary}[theorem]{Corollary}

\newtheorem{definition}[theorem]{Definition}
\title{Relations between connected and self-avoiding walks in a digraph}
\author{Thibault Espinasse\footnote{Universit\'e Lyon 1 Claude Bernard, laboratoire de math\'ematiques Camille Jordan} ~ and Paul Rochet\footnote{Universit\'e de Nantes, laboratoire de math\'ematiques Jean Leray, paul.rochet@univ-nantes.fr}}
\date{}
\newcommand{\sgn}{\operatorname{sgn}}
\newcommand{\id}{\operatorname{I}}
\newcommand{\tr}{\operatorname{tr}}
\newcommand{\adj}{\operatorname{adj}}

\newenvironment{proof}{\noindent \textit{Proof.}}{\hfill $\square$\bigskip}
\newenvironment{remark}{\noindent \textit{Remark.}}{\bigskip}

\begin{document}

\maketitle 
\begin{abstract} Walks in a directed graph can be given a partially ordered structure that extends to possibly unconnected objects, called hikes. Studying the incidence algebra on this poset reveals unsuspected relations between walks and self-avoiding hikes. These relations are derived by considering truncated versions of the characteristic polynomial of the weighted adjacency matrix, resulting in a collection of matrices whose entries enumerate the self-avoiding hikes of length $\ell$ from one vertex to another. 
\end{abstract}



\normalsize

\section{Introduction}

A directed graph $G=(V,E)$ is defined by a vertex set $V=\{v_1,...,v_N\}$ and set $E$ of ordered pairs of vertices representing the directed edges of $G$. Directed graphs, or digraphs, have been extensively used in the literature as mathematical models to describe actual phenomena such as social interactions \cite{carrington2005models}, road traffic \cite{han2012extended}, physical processes \cite{burioni2005random} or random walks \cite{blanchard2011random} among many others. In most models, it is convenient to allocate weights to each edge of the graph in order to incorporate some additional information. For example, a weight $\omega_{ij}$ between two nodes $v_i, v_j$ may serve defining the transition probability of a random walk, the speed limit or the type of road in the traffic network and so on. A finite graph on $N$ vertices is then characterized by its \textit{weighted adjacency matrix} $\mathsf{W}= (\omega_{ij})_{i,j=1,...,N}$, which accounts for the level of interaction between vertices. In this way, every edge of $G$ is identified with a weight $\omega_{ij}$.\\

An oriented walk on the digraph $G$ is defined as a succession of contiguous edges $\omega_{ij}$. Seeing the weights $\omega_{ij}$ as formal variables, a walk $w$ of length $\ell$ from $v_i$ to $v_j$ can be viewed as a degree $\ell$ monomial $w= \omega_{i i_1} \omega_{i_1 i_2} ... \omega_{i_{\ell-1} j}$. The weighted adjacency matrix then provides a practical tool to handle walks on the graph, as they can be derived from analytical transformations of $\mathsf{W}$. For instance, the $(i,j)$ entry of $\mathsf{W}^2$, given by $\mathsf{W}^2_{ij} = \sum_{k} \omega_{ik} \omega_{kj}$, enumerates all walks of length $2$ from $v_i$ to $v_j$. The introduction of the weighted adjacency matrix $\mathsf{W}$ to describe the walks on a graph goes back to the $60$'s. In \cite{harary1962determinant} and \cite{ponstein1966self}, the spectral properties of a graph are investigated via the determinant and characteristic polynomial of $\mathsf{W}$. Digraphs also provide a useful tool to compute the determinant and minors of sparse matrices, as discussed in \cite{maybee1989matrices}. For general results on spectral graph theory, we refer to \cite{chung1997spectral, cvetkovic1997eigenspaces}. \\

In \cite{ponstein1966self}, the author shows that the coefficients of the characteristic polynomial of $\mathsf{W}$ can be interpreted in term of the self-avoiding cycles in $G$. In this paper, we derive a similar result concerning self-avoiding hikes, defined as a generalization of walks to possibly unconnected sequences of edges. We construct a collection of polynomials of $\mathsf{W}$ whose entries enumerate the self-avoiding hikes of a given length from one vertex to another. The polynomials are obtained as Cauchy products of the characteristic polynomial coefficients with the sequence of successive powers of $\mathsf{W}$. \\

The analytical expression linking the self-avoiding hikes and the walks on the graph hides a deeper connection when considering each hike individually. Precisely, the relation can be investigated in the partially ordered set formed by the hikes. In this context, combinatorial properties arise when studying functions of the hikes in the reduced incidence algebra of this poset. In particular, we show that the number of different ways to travel a closed hike can be expressed in term of its self-avoiding divisors via a Mobius-like inversion on this poset. Another result on the decomposition of a walk into self-avoiding components is then derived. \\

The paper is organized as follows. Definitions are introduced in Section \ref{sec:2} as well as the preliminary result. Section \ref{sec:3} is devoted to the study of the different relations between self-avoiding hikes and walks, where many combinatorial properties are investigated. The results are verified on specific examples in Section \ref{sec:4}.

\section{Notations and preliminary results}\label{sec:2}

Let $G=(V,E)$ be a labeled directed graph, or \textit{digraph}, with finite vertex set $V = \{v_1,...,v_N \}$ and edge set $E$ which may contain loops. 
The adjacency matrix of $G$ is defined as the $N \times N$ matrix $\mathsf{A}$ with entries $a_{ij}$ equal to one if $v_i$ is connected to $v_j$ and zero otherwise. Because $G$ is directed, $\mathsf{A}$ may not be symmetric ($v_i$ can be connected to $v_j$ without $v_j$ being connected to $v_i$). In this paper, an edge always refers to a directed edge. 

The adjacency matrix can be used to derive numerous properties of a graph. For instance, the $(i,j)$ entry of $\mathsf{A}^\ell$ gives the number of walks of length $\ell$ from $v_i$ to $v_j$. When one is interested in each walk specifically, a useful tool is to allocate a weight, or variable, to each non-zero entry of $\mathsf{A}$. In this way, the digraph $G$ is characterized by its weighted adjacency matrix
$$\mathsf{W} = (\omega_{ij})_{i,j=1,...,N} $$
where the $\omega_{ij}$'s are real variables, setting $\omega_{ij}=0$ whenever there is no edge from $v_i$ to $v_j$. An edge of $G$ can then be identified with a non-zero variable $\omega_{ij}$ and a walk $w$ from $v_i$ to $v_j$ with the product $w = \omega_{i i_1} \omega_{i_1 i_2} ... \omega_{i_{\ell-1} j}$ of the edges composing it (two walks are thus considered equal if they are composed of the same edges, counted with multiplicity, regardless of their order). The walk $w$ is closed (a cycle) if $i=j$ and open (a path) otherwise. Moreover, $w$ is simple if it does not cross the same vertex twice, that is, if the indices $i,i_1,...,i_{\ell-1},j$ are mutually different, with the possible exception $i=j$ if $w$ is closed. Loops $\omega_{ii}$ and backtracks $\omega_{ij} \omega_{ji}$ are considered cycles of length $1$ and $2$ respectively. \\

The representation as monomials in the formal variables $\omega_{ij}$ provides a simple multiplicative structure on walks. The cycle-erasing procedure of Lawler \cite{lawler1987loop} shows that a walk from $v_i$ to $v_j$ can always be decomposed as the product of a simple walk from $v_i$ to $v_j$ and cycles, as illustrated in Figure \ref{fig:0}. However, the reverse is not true in general as the product of a simple walk and simple cycles might not be connected. In this paper, we define a new object, called hike, which extends the definition of a walk by relaxing the connectedness condition. 

\begin{definition} A hike from $v_i$ to $v_j$ is a monomial $h= \omega_{i_1 j_1}  ... \omega_{i_{\ell} j_\ell}$ that can be decomposed into the product of a simple walk from $v_i$ to $v_j$ and simple cycles.
\end{definition}

Properties of walks naturally extend to hikes. A hike is closed if it is a product of simple cycles and open otherwise. By convention, the trivial cycle $1$ is considered a closed hike of length $0$. Similarly as for walks, a hike is self-avoiding if it does not cross the same vertex twice. The connected components of a self-avoiding hike are simple and vertex-disjoint. In this setting, a walk can be viewed as a connected hike. \\

In the sequel, the length of a hike $h$ (its degree) will be denoted by $\ell(h)$ and its number of connected components by $n(h)$. A walk is a connected hike, that is, a hike $h$ such that $n(h)=1$. Moreover, we denote by $V(h)$ the set of vertices crossed by $h$ and $\vert V(h) \vert$ its cardinal. Remark that an open hike $h$ is self-avoiding if, and only if, $\vert V(h) \vert = \ell(h)+1$ while for a closed hike, self-avoiding is equivalent to $\vert V(h) \vert = \ell(h)$. \vspace{-0.4cm}

\begin{multicols}{2}

\begin{figure}[H]
	\centering
\hspace*{-0.6cm} \includegraphics[width=0.35\textwidth]{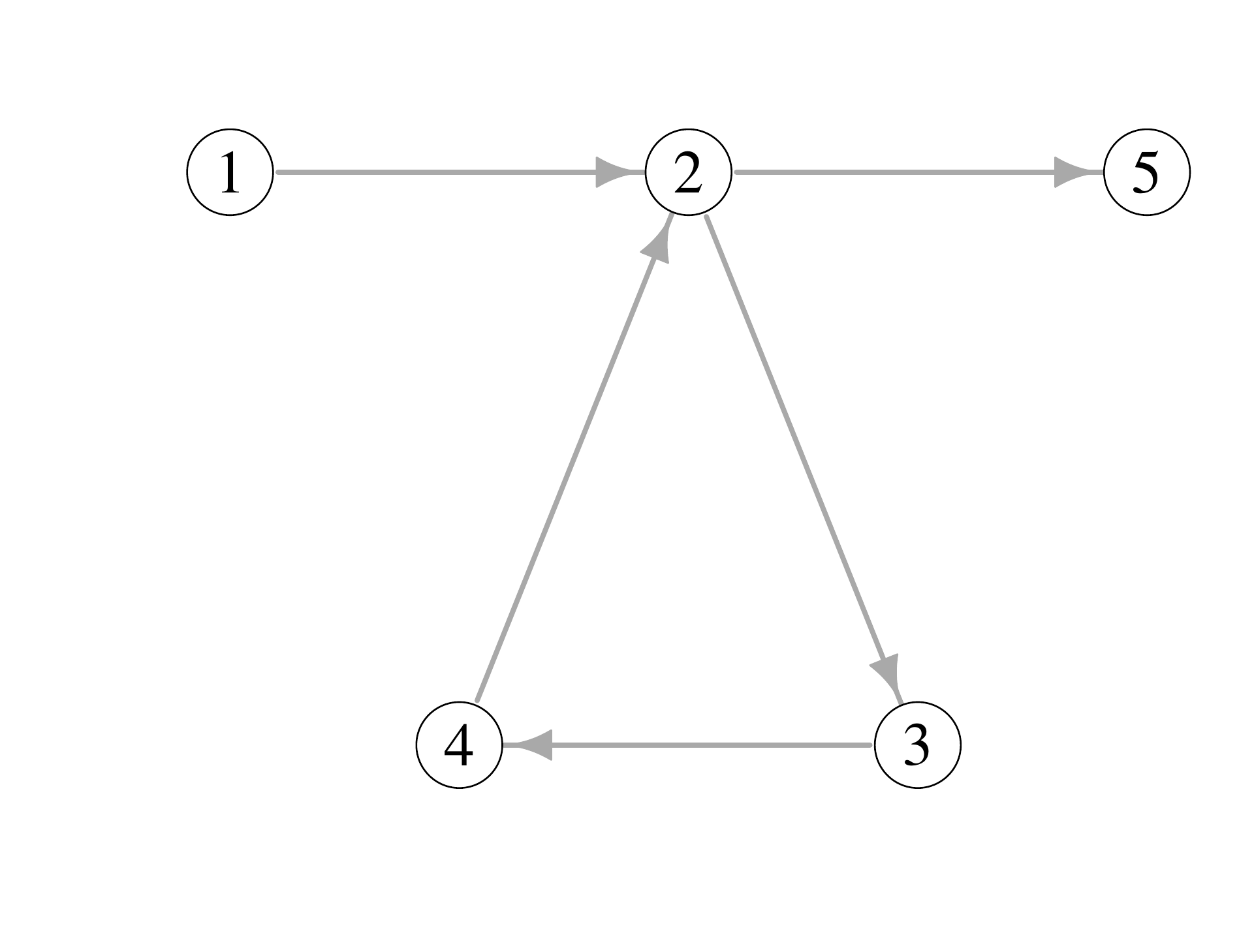} 
	\vspace{-0.6cm}
	\caption{\small{The walk $w= \omega_{12}\omega_{23}\omega_{34}\omega_{42}\omega_{25}$ from $v_1$ to $v_5$ is the product of the simple walk $\omega_{12}\omega_{25}$ and simple cycle $\omega_{23}\omega_{34}\omega_{42}$.}}
	\label{fig:0}
\end{figure}

\columnbreak

\begin{figure}[H]
	\centering
 \hspace*{1.1cm} \includegraphics[width=0.35\textwidth]{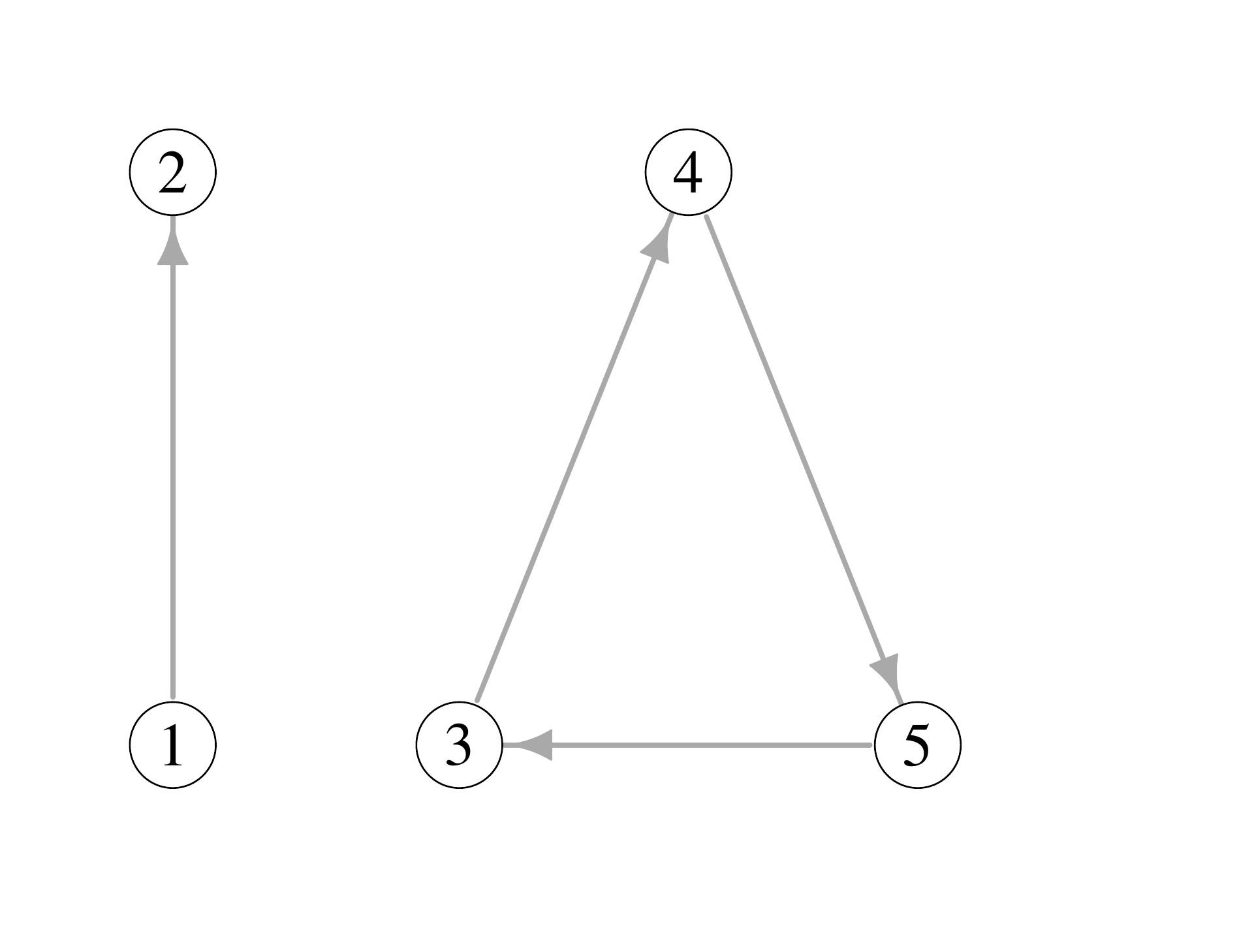}
	\vspace{-0.6cm}
	\caption{\small{The self-avoiding hike $h =\omega_{12} \omega_{34}\omega_{45}\omega_{53}$ from $v_1$ to $v_2$ contains $n(h)=2$ connected components: the path $\omega_{12}$ and the cycle $\omega_{34}\omega_{45}\omega_{53}$.}}
	\label{fig:0b}
\end{figure}

\end{multicols}

\noindent The following notations are used throughout the paper.
\begin{itemize}
  \item[-] $\mathcal W$ denotes the set of walks. 
  \item[-] $\mathcal H$ denotes the set of hikes (open and closed). 
  \item[-] $\mathcal C$ denotes the set of closed hikes.
	\item[-] $\mathcal S$ denotes the set of self-avoiding hikes. 
\end{itemize}
For these sets, we may specify the end-vertices in index and/or their length in exponent, e.g. $\mathcal H^\ell_{ij}$ is the set of hikes of length $\ell$ from $v_i$ to $v_j$. Remark that $\mathcal W$, $\mathcal C$ and $\mathcal S$ are subsets of $\mathcal H$, while $\mathcal S_{ii}$ and $\mathcal W_{ii}$ are subsets of $ \mathcal C$ for all $i=1,...,N$. The set of self-avoiding closed hikes is written as $\mathcal C \cap \mathcal S$. \\ 

It is known that the $\ell$-th power of $\mathsf{W}$ enumerates with multiplicity the walks of length $\ell$ on the graph. The $(i,j)$ entry of $\mathsf{W}^\ell$ is a homogenous degree $\ell$ polynomial attributing to each walk $w$ of length $\ell$, a coefficient $f_{ij}(w)$ counting the number of ways to write $w$ as a succession of contiguous edges starting from $v_i$ and ending at $v_j$ (the function $f_{ij}$ is computed on some examples in Figure \ref{fig:22}). The formal series associated to the functions $f_{ij}$ follows by the identity
\begin{equation}\label{fij} (\id - \mathsf{W} )^{-1}_{ij} = \big(1+\mathsf{W}+\mathsf{W}^2+.... \big)_{ij} = \sum_{w \in \mathcal W} f_{ij}(w) w,   \end{equation}
which holds whenever $||| \mathsf{W} ||| := \sup_{v \in \mathbb R^N, \Vert v \Vert=1} \Vert \mathsf{W} v \Vert < 1$, with $\Vert . \Vert$ the Euclidean norm in $\mathbb R^N$. Remark that if $w$ is an open walk, there is at most one couple $(i,j)$ for which $f_{ij}(w)$ is non-zero. This property is no longer verified for a closed walk $c$ in which case $f_{ii}(c)$ may take different positive values for different nodes $v_i \in V(c)$, as illustrated in Figure \ref{fig:22}. \vspace{-0.4cm}

\begin{figure}[H]
	\centering
		\includegraphics[width=0.54\textwidth]{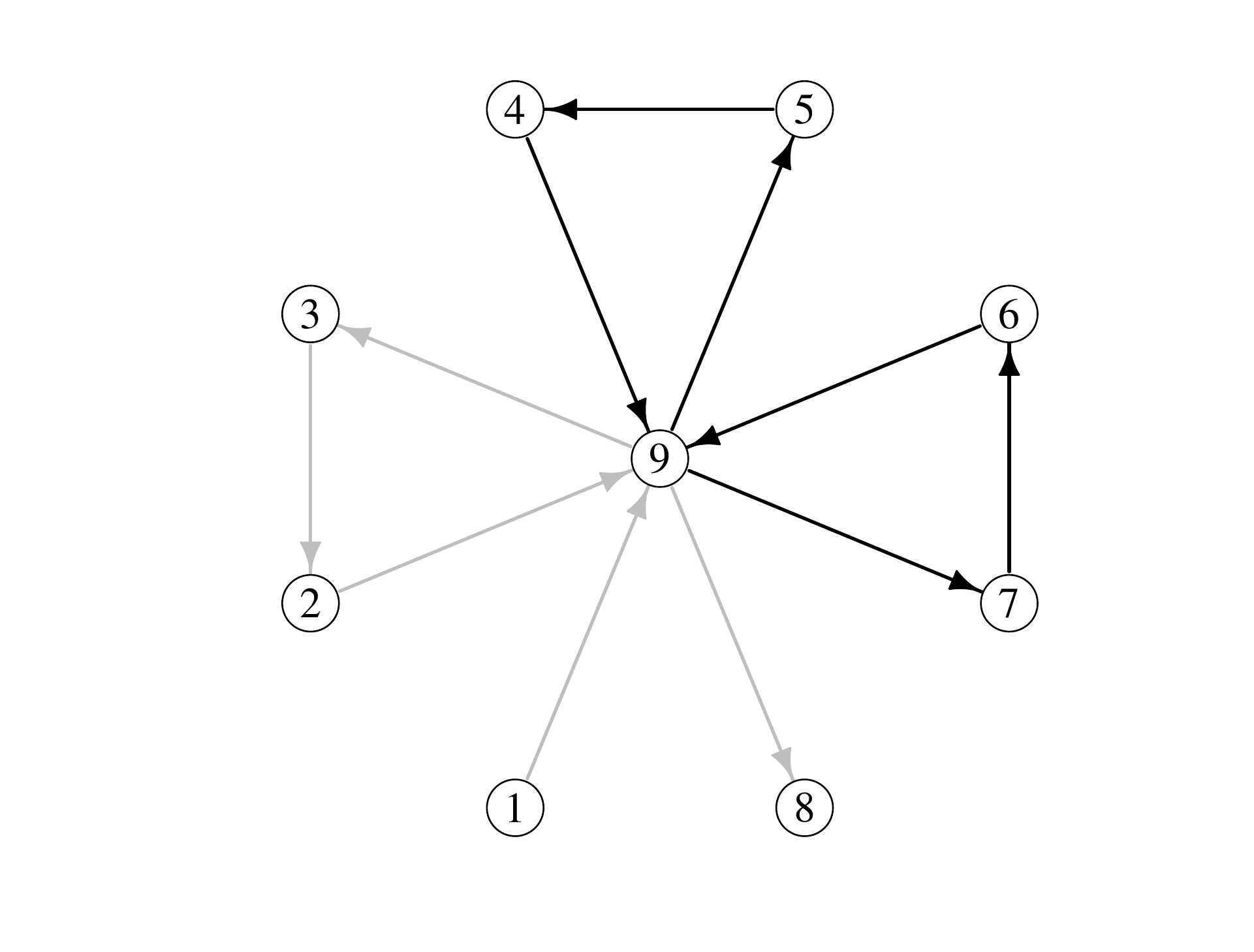}
		\vspace{-0.8cm}
		\caption{\small{The walk $w= \omega_{19}\omega_{93}\omega_{32}\omega_{29}\omega_{95}\omega_{54}\omega_{49}\omega_{97}\omega_{76}\omega_{69}\omega_{98}$ (gray and black edges) has $f_{18}(w)=6$ contiguous representations starting from $v_1$ and ending in $v_8$. Each representation relies on an ordering of the simple cycles $\omega_{93}\omega_{32}\omega_{29} $, $\omega_{95}\omega_{54}\omega_{49}$ and $\omega_{97}\omega_{76}\omega_{69}$ to travel $w$. The cycle $c=\omega_{95}\omega_{54}\omega_{49}\omega_{97}\omega_{76}\omega_{69}$ (black edges) has $f_{99}(c) =2$ contiguous representations starting from $v_9$ and $f_{ii}(c)=1$ contiguous representation for any other vertex $v_i \in V(c) \setminus \{v_9 \}$.}}
	\label{fig:22}
\end{figure}

Closed hikes also play an essential part in analytical graph theory. A simple reason is that a self-avoiding closed hike can be associated to a permutation $\sigma$ on a subset of $\{1,...,N\}$ by the relation $ c_\sigma = \prod_{v_i \in V(c_\sigma)} \omega_{i \sigma(i)}$. Using that the number of connected components $n(c_\sigma)$ is linked to the signature of $\sigma$ through the identity $\sgn(\sigma) = (-1)^{ N - n(c_\sigma)}$, we obtain an expression of the determinant of $\mathsf{W}$ by
$$ \operatorname{det}(\mathsf{W}) =  \sum_{\sigma} \sgn(\sigma) \ \omega_{1\sigma(1)}...\omega_{N\sigma(N)} = (-1)^N \sum_{c \in \mathcal C\cap \mathcal S^N} (-1)^{n(c)} c $$
where the first sum is taken over all permutations over $\{1,...,N\}$ and the second sum over all self-avoiding closed hikes of length $N$. A more general formula, given in Theorem 1 in \cite{ponstein1966self}, links the coefficients $\psi_k$ of the characteristic polynomial
\begin{equation*} \chi(\lambda) = \det \big(\lambda \id - \mathsf{W} \big) = \sum_{k = 0}^{N} \psi_k \lambda^{N-k}\end{equation*}
with the self-avoiding closed hikes of length $k$ by
\begin{equation}\label{psi} \psi_0=1, \ \ \psi_k = \sum_{c \in \mathcal C\cap \mathcal S^k} (-1)^{n(c)} c, \ \ k = 1,2,...  \end{equation}
the coefficient $\psi_k$ being trivially zero for $k >N$.\\

The coefficient $(-1)^{n(c)}$ is reminiscent of a Mobius function. In fact, the function $\mu$ defined over $\mathcal C$ by $\mu(1) = 1$ and
\begin{equation}\label{mu} \mu(c) := \left\{ \begin{array}{cl} (-1)^{n(c)} & \text{if $c$ is self-avoding} \\
0 & \text{otherwise,} \end{array} \right. \end{equation}
is the Mobius function of the trace monoid of a partially commutative version of closed hikes, studied in \cite{cartier1969} under the name circuit. The particular value of the characteristic polynomial
\begin{equation}\label{mobdet} \chi(1) = \det (\id - \mathsf{W} ) = \sum_{k=0}^N \psi_k = \sum_{c \in \mathcal C} \mu(c) c \end{equation}
will turn out to be particularly important for our purposes, as it gives the formal series associated to $\mu$. Although the definition of $\mu$ is restricted to closed hikes, some of its properties have direct repercussions on open hikes. This is due to the fact that an open hike $h$ between two different vertices $v_i, v_j$ can be expressed as a closed hike to which the edge $\omega_{ji}$ has been removed. Actually, one verifies easily the following equivalence for $i \neq j$
\begin{equation}\label{equiv1} h \in \mathcal H_{ij} \ \Longleftrightarrow \ h \omega_{ji} \in \mathcal C \end{equation}
where we recall that $\mathcal H_{ij}$ is the set of hikes from $v_i$ to $v_j$. This property provides a natural extension of $\mu$ to open hikes. For all directed pairs of vertices $(i,j)$, define
\begin{equation}\label{gij} \mu_{ij}(h) := - \mu(h \omega_{ji}) = \left\{ \begin{array}{cl}  (-1)^{n(h)+1} & \text{if $i \neq j$ and $h \in \mathcal S_{ij}$} \\ (-1)^{n(h)} & \text{if $i=j$, $h \in \mathcal C \cap \mathcal S$ and $v_i \notin V(h)$,} \\  
0 & \text{otherwise.} \end{array} \right. \end{equation}


If $i \neq j$, the function $\mu_{ij}$ only takes non-zero values for self-avoiding hikes from $v_i$ to $v_j$. In particular, $\mu_{ij}(1)=0$ and $\mu_{ij}(h)=1$ if $h$ is a simple path. On the other hand, $\mu_{ii}(h)$ is non-zero only if $h$ is closed and does not cross $v_i$. We have for instance $\mu_{ii}(1)= \mu(\omega_{ii}) = 1$ (the convention $\mu_{ii}(1)=1$ instead of $-1$ justifies the minus sign in the definition of $\mu_{ij}$). Remark moreover that $\mu_{ij}(h)$ is null for all hikes $h$ of length $\ell(h) \geq N$ since a self-avoiding hike on the graph has maximal length $N$. \\

The main result of this paper (Theorem \ref{th:1}) exhibits a duality between the functions $f_{ij}$ counting the number of connected representations of a hike and the functions $\mu_{ij}$ whose supports are restricted to self-avoiding hikes. The duality is actually a consequence of Lemma \ref{lemma}, which expresses the matrix formal series $\mathsf{M} = (m_{ij})_{i,j=1,...,N}$ of $\mu_{ij}$, defined by
\begin{equation}\label{x} m_{ij} :=  \sum_{h \in \mathcal H} \mu_{ij}(h) h , \ \ \ i,j=1,...,N, \end{equation}
in function of the powers of $\mathsf{W}$.

\begin{lemma}\label{lemma} If $||| \mathsf{W} ||| <1$,
\begin{equation}\label{main} \mathsf{M} = \sum_{\ell \geq 0} \sum_{k=0}^{\ell} \psi_k \mathsf{W}^{\ell - k}. \end{equation}
where the $\psi_k, k=0,1,...$ are the coefficients of the characteristic polynomial defined in \eqref{psi}.
\end{lemma}

\begin{proof}  The proof relies on the adjugate identity $\adj (\mathsf{B} ) = \det (\mathsf{B})  \mathsf{B}^{-1}$ applied to $\mathsf{B} = \id - \mathsf{W}$. Recall that for a matrix $\mathsf{B} = (b_{ij})_{i,j=1,...,N}$, the adjugate of $B$ is defined as
\begin{equation}\label{adj} \adj(\mathsf{B}) = \big( \det ( \mathsf{B}^{(ji)}) \big)_{i,j=1,...,N}, \end{equation}
where $\mathsf{B}^{(ji)}$ is the matrix obtained by setting $b_{ji}=1$, $b_{ki}=0$ for $k\neq j$ and $b_{jk} = 0$ for $k \neq i$ in $\mathsf{B}$. Combining Equation \eqref{mobdet} with the identity $\adj (\id - \mathsf{W} ) = \det (\id - \mathsf{W}) \ (\id - \mathsf{W} )^{-1}$ gives for $||| \mathsf{W} |||<1$
\begin{equation}\label{adj:normal} \adj (\id - \mathsf{W} ) = \sum_{k=0}^N \psi_k \times \sum_{k \geq 0} \mathsf{W}^k  = \sum_{\ell \geq 0} \sum_{k=0}^\ell \psi_k \mathsf{W}^{\ell -k}  \end{equation}
with $\psi_k=0$ for $k >N$. It remains to show that $\mathsf{M} = \adj ( \id - \mathsf{W})$. We proceed entry-wise, first considering the case $i=j$. By construction of the adjugate matrix in \eqref{adj} for $\mathsf{B}=\id -\mathsf{W} $, the conditions $b_{ii} =1$ and $b_{ki} =b_{ik} = 0$ for $k \neq i$ correspond to setting $\omega_{ki}= \omega_{ik}= 0$ for all $k$. Plugging these values into $\det (\id - \mathsf{W} ) = \sum_{c \in \mathcal C} \mu(c) c$ sets to zero every cycle crossing $v_i$, yielding
\begin{equation}\label{adj-ii} \big(\adj(\id-\mathsf{W}) \big)_{ii} 
= \sum_{\substack{c \in \mathcal C \\ i \notin V(c)}} \mu(c) c = \sum_{h \in \mathcal H} \mu_{ii}(h) h. \end{equation}
Now consider the case $i \neq j$. Going back to Equation \eqref{adj:normal}, we see that the $(i,j)$ entry of $\adj(\id - \mathsf{W})$ satisfies
$$ \big(\adj(\id - \mathsf{W})\big)_{ij} = \sum_{k=0}^N \psi_k \times \sum_{k \geq 0} \mathsf{W}^k_{ij} = \sum_{c \in \mathcal C} \mu(c) c \times \!\! \sum_{w \in \mathcal W} f_{ij}(w) w = \!\!\!\!\!\!   \sum_{(c,w) \in \mathcal C \times \mathcal W}  \!\!\!\!\!\! \mu(c) f_{ij}(w) c w.   $$
The right-hand side is a sum over hikes of the form $h = c w$ with $c$ a self-avoiding closed hike and $w$ a walk from $v_i$ to $v_j$. By \eqref{adj}, the left-hand side is obtained by plugging the values 
$\omega_{ji}=-1$, $\omega_{ki}=\omega_{jk} = 0$ for $k \neq i,j$ and $\omega_{ii} = \omega_{jj} = 1$ into $\det (\id - \mathsf{W}) = \sum_{c \in \mathcal C} \mu(c) c$. Since only hikes from $v_i$ to $v_j$ remain (by identification with the right-hand side), we deduce in view of \eqref{equiv1}, 
\begin{equation}\label{adj-ij} \big(\adj(\id-\mathsf{W}) \big)_{ij} = - \sum_{h \in \mathcal H_{ij}} \mu(h \omega_{ji}) h = \sum_{h \in \mathcal H} \mu_{ij}(h) h. \end{equation}
Thus, $\adj(\id - \mathsf{W}) = \mathsf{M}$. 
\end{proof}\\

The restriction $\mathsf{M}^{(\ell)} $ of $\mathsf{M}$ to hikes of length $\ell$ satisfies, by identifying the terms of equal degrees in Lemma \ref{lemma},
\begin{equation}\label{trace}  \mathsf{M}^{(\ell)}  =  \sum_{k=0}^{\ell} \psi_k \mathsf{W}^{\ell - k}= \bigg( \sum_{\substack{h \in \mathcal H^\ell}} \mu_{ij}(h) h \bigg)_{i,j=1,...,N}.   \end{equation}
In particular, $\mathsf{M}^{(\ell)}$ commutes with $\mathsf{W}$ for all $\ell \geq 0$. If the digraph contains few self-avoiding hikes of length $\ell$, the matrix $\mathsf{M}^{(\ell)}$ may have many zero entries. The construction gives in this case a non-trivial sparse matrix in the commutant of $\mathsf{W}$. This kind of problems has some applications in practice, for instance in the study of random processes. In \cite{fytp14}, the authors investigate conditions under which the transition kernel of a finite state Markov chain observed at random times can be estimated consistently. They show that sparsity conditions, arising from particular state transitions known to be impossible, suffice to recover the transition kernel when it commutes with a certain matrix for which an estimator is available. The identifiability conditions in this model rely on the existence of a sparse matrix in the commutant. Similar problems and applications are studied in \cite{inamura2006,MR0468086,MR668189}.\\

\noindent From the definition of $\mu_{ij}$, one verifies easily that $\widetilde{\mathsf{M}}^{(\ell)} := \psi_\ell \id - \mathsf{M}^{(\ell)} $ satisfies
\begin{equation}\label{mtilde} \widetilde{\mathsf{M}}^{(\ell)}_{ij} = - \sum_{k=0}^{\ell-1} \psi_k \mathsf{W}^{\ell - k}_{ij} = \sum_{\substack{h \in \mathcal S^\ell_{ij}}} (-1)^{n(h)}  h \end{equation}
for all $i,j =1,...,N$. This shows that the matrix enumerating, up to the coefficient $(-1)^{n(h)}$, the self-avoiding hikes of length $\ell$ for all pairs of vertices can be obtained as a polynomial of $\mathsf{W}$. Because $\mu_{ij}(h)$ is trivially zero when $\ell(h) = N$, one recovers Cayley-Hamilton's theorem by setting $\ell=N$ in Equation \eqref{trace}. The general case also gives a direct proof of the identity
\begin{equation}\label{tracebb} \psi_\ell = - \frac 1 \ell \sum_{k=0}^{\ell-1} \psi_k \tr \big(\mathsf{W}^{\ell-k} \big), \end{equation}
for which different proofs can be found in \cite{zadeh1976linear} and \cite{hong2008trace}. To prove it using Lemma \ref{lemma}, observe that \eqref{mtilde} yields
\begin{align} \tr \Big(\widetilde{\mathsf{M}}^{(\ell)}\Big) = \sum_{i=1}^N \sum_{\substack{c \in \mathcal C \cap \mathcal S^\ell_{ii}}} (-1)^{n(c)} c = \ell \sum_{c \in \mathcal C \cap \mathcal S^\ell} (-1)^{n(c)} c  = \ell \psi_\ell, \nonumber
\end{align}
noticing that each self-avoiding closed hike $c$ appears exactly $\ell = \ell(c)$ times when summing over $i$. Hence, Equation \eqref{tracebb} follows directly by computing the trace on both sides of the equality
$$ \widetilde{\mathsf{M}}^{(\ell)} = - \sum_{k=0}^{\ell-1} \psi_k \mathsf{W}^{\ell-k}.$$

\section{Incidence algebra on hikes}\label{sec:3}


Because $\mathcal C$ is stable by multiplication, it forms a monoid with the empty cycle $1$ as identity element. A natural partial order on $\mathcal C$ arises from division: $d \in \mathcal C$ divides $c$, denoted by $d | c$, if there exists $c' \in \mathcal C$ such that $c = d c'$. The closed hikes ordered by division form a locally finite partially ordered set, or \textit{poset}.

\begin{definition} The reduced incidence algebra of closed hike is the algebra of real valued functions on $\mathcal C$, endowed with the Dirichlet convolution, defined for $f,g : \mathcal C \to \mathbb R$ by
$$  f * g (c) = \sum_{d | c} f(d) g \Big( \frac c d \Big) \ , \ c \in \mathcal C.  $$
\end{definition} 

In this definition, the sum is taken over all divisors $d \in \mathcal C$ of $c$, including the trivial cycle $1$ and $c$ itself. The reduced incidence algebra is isomorphic to the algebra of formal series, endowed with multiplication. Indeed, for two functions $f,g: \mathcal C \to \mathbb R$
$$ \bigg( \sum_{c \in \mathcal C} f(c) c \bigg) \times \bigg( \sum_{c \in \mathcal C} g(c) c \bigg) = \sum_{c \in \mathcal C} f*g(c) c . $$
It follows that the Dirichlet convolution is associative, commutative and distributive over addition. The function $\delta$ defined on $\mathcal C$ by $\delta(1)=1$ and $\delta(c) =0 $ for all $c \neq 1$ is the identity element for this operation as we have, for any function $f$ on $\mathcal{C}$, $f*\delta = \delta*f = f$. We refer to \cite{rota1987foundations} for a more comprehensive study on this subject. \\

Extending this structure of open hikes is slightly more complicated. The set of hikes $\mathcal H$ is not stable by multiplication and for this reason, defining a division over $\mathcal H$ leads to some difficulties. However, hikes are stable by multiplication with a closed hike so that $(\mathcal H,.)$ forms an act over $(\mathcal C,.)$. Thus, the division by a closed hike can be extended to open hikes: for $h \in \mathcal H$, $d$ divides $h$ if $d \in \mathcal C$ and $h = dh'$ for some $h' \in \mathcal H$. Similarly, the Dirichlet convolution can be extended to $f : \mathcal C \to \mathbb R$ and $g: \mathcal H \to \mathbb R$ by
$$  f * g (h) = \sum_{d | h} f(d) g \Big( \frac h d \Big) \ , \ h \in \mathcal H.  $$
Here again, the sum is taken only over closed divisors of $h$. Remark that $f$ and $g$ cannot be permuted in this expression, unless $h \in \mathcal C$. Interesting combinatorial properties arise from the poset structure of hikes, by considering each hike individually in Equation \eqref{main}. The duality between walks and self-avoiding hikes becomes apparent in the next theorem. We extend the definition of $f_{ij}$ to $\mathcal H$ by setting $f_{ij}(h)=0$ whenever $h$ is not connected.

\begin{theorem}\label{th:1} For all $i,j=1,...,N$, $\mu_{ij} = \mu * f_{ij}$.
\end{theorem}

\begin{proof} From Lemma \ref{lemma}, we know that whenever $ ||| \mathsf{W} ||| <1$,
$$ \mathsf{M} = \det (\id - \mathsf{W}) \ (\id - \mathsf{W} )^{-1}.$$ 
From \eqref{fij} and \eqref{gij} and using that $\det (\id - \mathsf{W}) = \sum_{c \in \mathcal C} \mu(c) c$, we obtain for all $i,j=1,...,N$,
$$ \sum_{h \in \mathcal H} \mu_{ij}(h) h  = \sum_{c \in \mathcal C} \mu(c) c  \times \sum_{h \in \mathcal H} f_{ij}(h) h = \sum_{h \in \mathcal H} \mu*f_{ij}(h) \ h.  $$
The result follows by identification.
\end{proof}

This theorem reveals a somewhat unexpected relation between the function $\mu_{ij}$ which only takes non zero values for self-avoiding hikes and $f_{ij}$ whose support contains only walks. Taking a closer look, the result is not surprising if $h$ is self-avoiding. Indeed, a self-avoiding hike $h$ from $v_i$ to $v_j$ has exactly one divisor $d \in \mathcal C$ such that $h/d$ is a walk from $v_i$ to $v_j$. Thus, the convolution $\mu * f_{ij}$ is calculated over only one non-zero element and the equality is easily verified in this case. The result is actually more interesting if $h$ is not self-avoiding as it yields in this case the non-trivial identity
$$ \forall h \in \mathcal H \setminus \mathcal S \ , \ \sum_{\substack{d | h}} \mu(d) f_{ij}\Big(\frac{h}{d}\Big) =0. $$

Clearly, the function $\mu$ is a key feature to understand the combinatorial properties of this poset. The fact that $\mu(1) = 1 \neq 0$ makes it invertible through the Dirichlet convolution, and its inverse $\beta: \mathcal C \to \mathbb R$ is the unique function characterized by $\mu * \beta = \beta * \mu = \delta$. A reversed relation, expressing $f_{ij}$ in function of $\mu_{ij}$ can then be derived easily, noticing that
$$ f_{ij} = (\beta * \mu) * f_{ij} = \beta *(\mu * f_{ij} ) = \beta * \mu_{ij}.    $$ 
This relation turns out to be particularly important for our purposes, as we show that $\beta$ satisfies interesting properties. In particular, we establish in the next proposition an expression of $\beta(c)$ that involves the number of appearances of each edge and vertex in $c$. Let $\tau_{ij}(c)$ denote the multiplicity of $\omega_{ij}$ in $c$ and $\tau_{i}(c) = \sum_{j=1}^N \tau_{ij}(c)$ the number of edges starting from $v_i$ in $c$ (counted with multiplicity).

\begin{theorem}\label{betaexpl} The function $\beta$ satisfies for all $c \in \mathcal C$, 
$$ \beta(c) = \prod_{i =1}^N \frac{\tau_{i}(c)!}{\tau_{i1}(c)! \times ... \times \tau_{iN}(c)!}.  $$

\end{theorem}

\begin{proof} We will prove that the function $\beta$ as defined in the theorem is the inverse of $\mu$ through the Dirichlet convolution. The case $c=1$ being trivial, we take $c \neq 1$. Since $\mu(d) =0$ if $d$ is not self-avoiding, we have
$$ \mu * \beta (c) = \sum_{d | c} \mu(d) \beta \Big( \frac c d \Big)  = \sum_{\substack{d | c \\ d \in \mathcal C \cap \mathcal S}} \mu(d) \beta \Big( \frac c d \Big).  $$
For $d \neq 1$ a self-avoiding divisor of $c$, denote by $\sigma_d$ the permutation over $V(d)$ associated to $d$, i.e. such that $ d = \prod_{v_i \in V(d)} \omega_{i \sigma_d(i)}$. Clearly, $\tau_{ij}(c/d) = \tau_{ij}(c)$ if $v_i \notin V(d)$, while for $v_i \in V(d)$, we have $\tau_i(c/d) = \tau_i(c) -1$ and
$$ \tau_{ij} \Big( \frac c d \Big) = \left\{ \begin{array}{cl} \tau_{ij}(c) - 1 & \text{ if } j= \sigma_d(i) \\  \tau_{ij}(c) & \text{ otherwise.} \end{array} \right. $$ 
It follows
\begin{eqnarray*} \beta \Big( \frac c d \Big) = \prod_{i=1}^N \frac{\tau_{i}(c/d)!}{\tau_{i1}(c/d)! \times ... \times \tau_{iN}(c/d)!} = \prod_{v_i \in V(d)} \frac{\tau_{i \sigma_d(i)}(c)}{\tau_{i}(c)} \times \beta (c).
\end{eqnarray*}
We get, including the case $d=1$
$$ \sum_{d | c} \mu(d) \beta \Big( \frac c d \Big) = \beta(c) \bigg( 1 + \sum_{\substack{d | c \\ d \in \mathcal C \cap \mathcal S \setminus \{ 1 \}} } \mu(d) \prod_{v_i \in V(d)} \frac{\tau_{i \sigma_d(i)}(c)}{\tau_{i}(c)} \bigg).  $$
Recall that for $d \in \mathcal C \cap \mathcal S^k$, $\mu(d) = (-1)^k \sgn(\sigma_d)$. Let $\ell = \vert V(c) \vert$ (the number of different vertices in $c$), the previous equality becomes, regrouping the divisors with equal lengths,
$$ \sum_{d | c} \mu(d) \beta \Big( \frac c d \Big) = \beta(c) \bigg( 1 +  \sum_{k=1}^{\ell} (-1)^k \sum_{\substack{d | c \\ d \in \mathcal C \cap \mathcal S^k}} \sgn(\sigma_d)  \prod_{v_i \in V(d)}  \frac{\tau_{i \sigma_d(i)}(c)}{\tau_{i}(c)} \bigg).  $$
Now consider the $\ell \times \ell$ matrix $\mathsf{B}(c)$ with entries $\tau_{ij}(c)/\tau_{i}(c)$ for $v_i,v_j \in V(c)$. By identifying each self-avoiding divisor $d$ of $c$ with its corresponding permutation $\sigma_d$, we recognize in the above expression the characteristic polynomial of $\mathsf{B}(c)$ taken at $\lambda =1$,
 $$ 1+ \sum_{k=1}^{\ell} (-1)^k \sum_{\substack{d | c \\ d \in \mathcal C \cap \mathcal S^k}} \sgn(\sigma_d)  \prod_{i \in d}  \frac{\tau_{i \sigma_d(i)}(c)}{\tau_{i}(c)} = \det( \id - \mathsf{B}(c)). $$
Since $\mathsf{B}(c)$ is a stochastic matrix, $\det( \id - \mathsf{B}(c)) =0$ which ends the proof.
\end{proof}

The coefficient $\beta(c)$ corresponds to the number of arrangements of the edges in $c$, regrouped by their initiating vertex. Indeed, the multinomial coefficient
$$ \frac{\tau_{i}(c)!}{\tau_{i1}(c)! \times ... \times \tau_{iN}(c)!} $$
counts the ways of ordering the edges initiating from $v_i$ in $c$, accounting for their multiplicity $\tau_{ij}(c)$. Considering all configurations for each vertex in $c$ recovers the coefficient $\beta(c)$. So, $\beta$ enumerates the different ways to travel along a closed hike. \\

This result points out some interesting properties of $\beta$, most of which are not straightforward from its initial definition as the inverse of $\mu$. The first immediate consequence is that $\beta$ is positive. Secondly, $\beta(c)$ is equal to one if $c$ is a self-avoiding closed hike. This condition is sufficient but not necessary, as we have for instance $\beta(c^2) =1$ as soon as $\beta(c)=1$. A third consequence is that $\beta$ is non-decreasing with respect to multiplication, which can be stated formally as: $\forall c_1, c_2 \in \mathcal C$, $\beta(c_1 c_2) \geq \max \{ \beta(c_1), \beta(c_2) \}$. Finally, $\beta$ is multiplicative over decompositions on disjoint closed hikes. Indeed, if $c$ can be written as the product of say $p \geq 2$ mutually vertex-disjoint components $c_1,...,c_p \in \mathcal C$, then $\beta(c) = \beta(c_1) ... \beta(c_p)$. This property is reminiscent of the multiplicity of arithmetic functions over coprime integers (see for instance \cite{apostol1990modular}). In this framework, two closed hikes $c_1,c_2$ can be considered coprime if they share no common vertex. The multiplicity property of $\beta$ is then inherited from the multiplicity of its inverse $\mu$.  \\

\begin{remark} The function $\mu$ is the Mobius function (i.e. the inverse of the constant function equal to $1$) of the trace monoid of circuits described in \cite{cartier1969}. Circuits correspond to partially commutative versions of closed hike so that, in our fully commutative framework, circuits composed of the same edges are seen as the same object. Thus, $\beta(c)$ counts the number of circuits composed of the same edges as $c$. 
\end{remark}

A different expression for $\beta$ can be derived from the inverse relation in Lemma \ref{lemma}, writing $\mathsf{W}^\ell$ in terms of the $\mathsf{M}^{(k)},k=0,1,...,\ell$. This result is given as a corollary. 
\begin{corollary}\label{inv} If $||| \mathsf{W} ||| <1$, then for $\ell \in \mathbb N$,
\begin{equation}\label{coro} \mathsf{W}^{\ell} = \sum_{k=0}^{\ell} \phi_k \mathsf{M}^{(\ell-k)}, \end{equation}
where the coefficients $\phi_0$, $\phi_1$,...,$\phi_N$ are defined by
\begin{equation}\label{phi} \phi_0= 1 \ , \ \phi_{k} = \sum_{k_1+...+k_p=k} (-1)^p \psi_{k_1} ... \psi_{k_p}, \ k = 1,...,N. \end{equation}
\end{corollary}
Before proving the result, let us clarify that the $\phi_k$'s are defined by taking the sum over all compositions of $k$, that is, all positive tuples $(k_1,...,k_p)$ such that $k_1+...+k_p=k$, for all $p=1,...,k$ (two tuples composed of the same integers $k_1,...,k_p$ but in different orders are to be counted twice). \\

\begin{proof} This holds if $\mathsf{W}$ is the null matrix with the convention $\mathsf{W}^0=\id$ in this case. For $\mathsf{W} \neq 0$, Equation \eqref{adj:normal} yields
$$ \mathsf{M} = \sum_{\ell \geq 0} \mathsf{M}^{(\ell)} =  \sum_{k =0}^N \psi_k \times \sum_{k \geq 0} \mathsf{W}^k \ \Longleftrightarrow \  \sum_{\ell \geq 0} \mathsf{W}^\ell = \frac 1 {\sum_{k =0}^N \psi_k } \ \sum_{k \geq 0} \mathsf{M}^{(k)}.$$
We use the formal series expansion
$$ \frac 1 {\sum_{k =0}^N \psi_k } = \frac 1 {1+ \sum_{k =1}^N \psi_k } = \sum_{p\geq 0} (-1)^p \Big( \textstyle\sum_{k =1}^N \psi_k \Big)^p.  $$
By regrouping the terms of equal degree, we get
\begin{equation}\label{inv-formal} \frac 1 {\sum_{k =0}^N \psi_k } = 1+ \sum_{k \geq 1} \ \  \sum_{ k_1+...+k_p=k} (-1)^p \psi_{k_1} ... \psi_{k_p} = \sum_{k \geq 0} \phi_k.  \end{equation}
Hence,
$$  \sum_{\ell \geq 0} \mathsf{W}^\ell = \sum_{k \geq 0} \phi_k  \times \sum_{k \geq 0} \mathsf{M}^{(k)} = \sum_{\ell \geq 0} \sum_{k =0}^\ell  \phi_k \mathsf{M}^{(\ell -k)}, $$
and the result follows by identification.
\end{proof}

Like the $\psi_k$'s, the coefficients $\phi_k$ are homogenous polynomials of degree $k$ in the $\omega_{ij}$'s. While $\psi_k$ only involves the self-avoiding closed hikes, $\phi_k$ depends on all the closed hikes of length $k$ on the digraph. The formal series inversion in Equation \eqref{inv-formal} actually corresponds to the inversion of the Dirichlet convolution when identifying each closed hike. This means in particular that the coefficient $\phi_k$ can be expressed as
\begin{equation}\label{beta} \phi_k = \sum_{ c \in \mathcal C^k} \beta(c) c. \end{equation}
One can verify this formula directly from the formal series multiplication
$$ 1 =  \sum_{k \geq 0} \phi_k \times \sum_{k \geq 0} \psi_k  = \sum_{c \in \mathcal C} \beta(c) c \times \sum_{ c \in \mathcal C} \mu(c) c = \sum_{c \in \mathcal C} \beta * \mu (c) c  $$
recovering exactly the formal series version of the equality $\delta = \beta * \mu$. By combining Equations \eqref{phi} and \eqref{beta}, we deduce a new expression of $\beta(c)$ for $c \neq 1$:
\begin{align}\label{beta2} \beta (c) = \sum_{p \geq 1} \ \  \sum_{c_1 ... c_p = c} (-1)^p \mu(c_1) ... \mu(c_p)  = \sum_{s_1 ... s_p = c} (-1)^{n'(s_1) + ... n'(s_p)} \end{align}
setting $n'(.) = n(.) + 1$, where the final sum is taken over all $p$-tuples $(s_1,...,s_p)$ of non-empty self-avoiding closed hikes such that $s_1...s_p = c$. This equality provides an expression of $\beta(c)$ involving the different decompositions of $c$ into self-avoiding closed hikes. While this expression is presumably less practical than the previous one, it induces nevertheless interesting consequences from a combinatorial point of view, which are discussed in Section \ref{sec:4}. \\

We now come to our final result, which expresses the multiplicity of an open walk in terms of its decompositions into self-avoiding components. This result will be illustrated on some examples in Section \ref{sec:4}.

\begin{theorem}\label{walk-covering} Let $h$ be a non-empty hike from $v_i$ to $v_j$,  
$$ f_{ij}(h) = \sum_{\substack{s_1 ... s_k p=h}} (-1)^{n'(s_1)+...+n'(s_k) + n'(p)} $$ 
setting $n'(.)=n(.)+1$, where the sum is taken over all self-avoiding decompositions of $h$, i.e., all $(k+1)$-tuples $(s_1,...,s_k,p) \in (\mathcal C \cap \mathcal S \setminus \{ 1 \})^k \times \mathcal S_{ij}$ with $ k \geq 0$ such that $ s_1 ... s_k p = h$.
\end{theorem}

\begin{proof} 
Plug the expression of $\beta$ in Equation \eqref{beta2} into $f_{ij}(h) = \sum_{d | h} \beta(d) \mu_{ij}(h/d)$.
For $i\neq j$, the fact that $\mu_{ij}(h/d) = 0$ for $h/d \notin \mathcal S_{ij}$ simplifies into
\begin{eqnarray*} f_{ij}(h) = \sum_{\substack{d=s_1 ... s_k | h \\ h/d \in \mathcal S_{ij}} } (-1)^{n'(s_1) + ... + n'(s_k)}  \times  (-1)^{n'(h/d)} = \!\!\!\! \sum_{s_1 ... s_k p =h} (-1)^{n'(s_1) + ... + n'(s_k)+ n'(p)} \end{eqnarray*}
for $p=h/d$, thus recovering the result. For $i = j$, we use that $ \mu_{ii}(c) = \mu(c)\mathds 1 \{ v_i \notin V(c) \}$,
\begin{eqnarray*}  f_{ii}(h) = \sum_{\substack{d | h} } \beta(d) \mu_{ii}\Big( \frac h d \Big) = \sum_{\substack{d | h} }\beta(d) \mu \Big( \frac h d \Big)  - \sum_{\substack{d | h \\ v_i \in V(h/d)} } \beta(d) \mu\Big( \frac h d \Big). 
\end{eqnarray*}
The first term of the right-hand side is $-\beta* \mu(h)$ which is zero for all $h \neq 1$. The result follows by using the expression of $\beta$ given in Equation \eqref{beta2}, similarly as in the case $i \neq j$.
\end{proof}

\section{Examples}\label{sec:4}

In this section, the functions $f_{ij}, \mu_{ij}, \beta$ and $\mu$ are computed on some examples. For ease of comprehension, we start with explicit simple cases before considering more general structures in the final examples.\\

\noindent \textit{Example 1.} 
Let us begin with the graph represented in Figure~\ref{fig:2} which contains only two disjoint simple cycles.\vspace{-0.8cm}

\begin{figure}[H]
	\centering
		\includegraphics[width=0.46\textwidth]{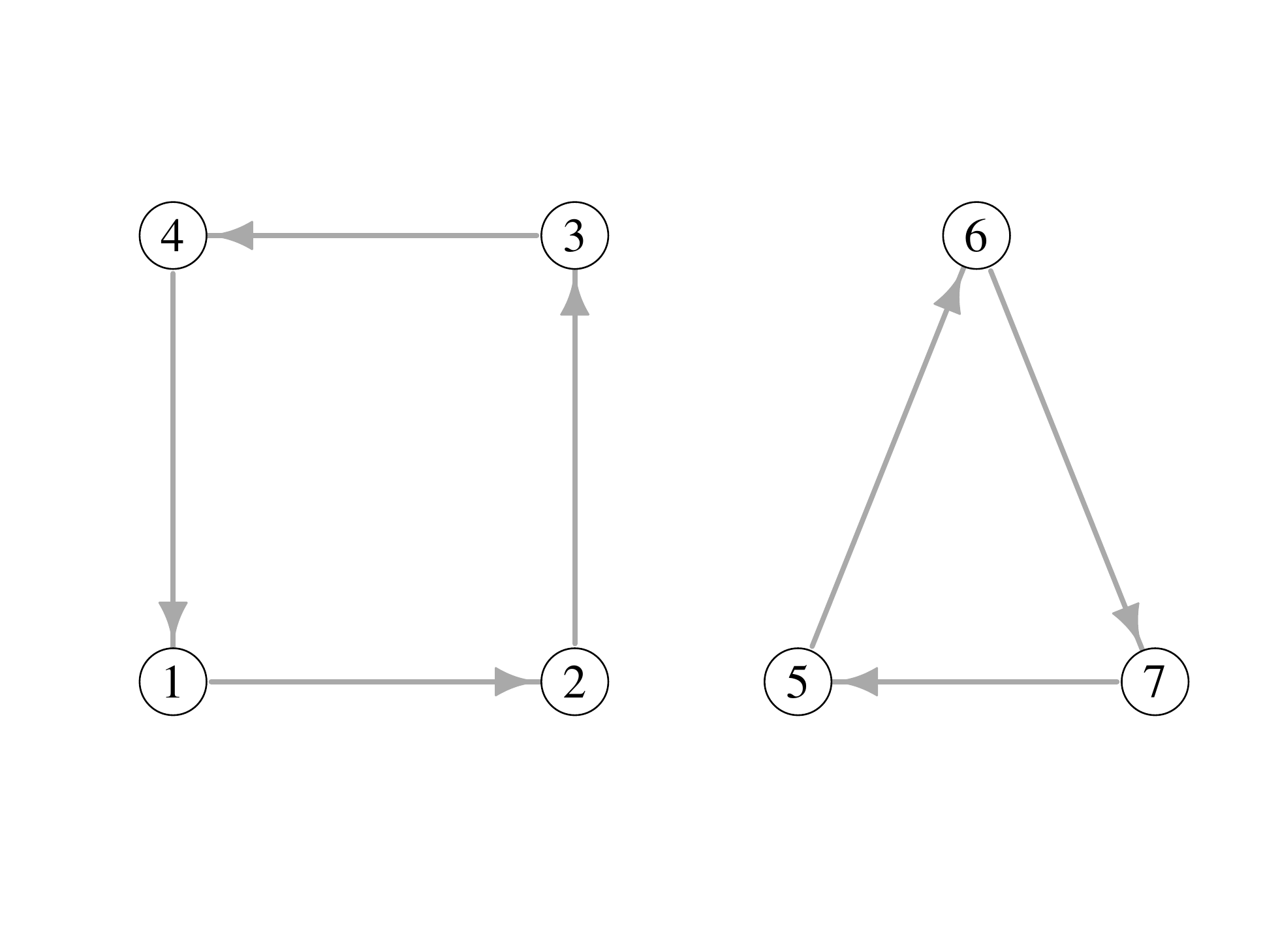}
\vspace{-1.6cm}
	\caption{\small{Disjoint simple cycles.}}
  \label{fig:2}
\end{figure}

\noindent The hike covering the whole graph is the closed hike $h_1 =\omega_{12}\omega_{23}\omega_{34}\omega_{41}\omega_{56}\omega_{67}\omega_{75} $ (recall that the order is not important). We obtain directly $f_{ij}(h_1)=0$ (because $h_1$ is not connected) and $\mu_{ij}(h_1) = 0$ (because $h_1$ crosses every vertex) for all $i,j =1,...,7$. Moreover, the definitions of $\mu$ and $\beta$ give in this case $\mu(h_1) = \beta(h_1) = 1$. To check the equalities $\mu*f_{ii} = \mu_{ii}$ and $\beta*\mu_{ii} = f_{ii}$, the calculations are straightforward, since the only closed divisors of $h_1$ are $h_1, c_1=\omega_{12}\omega_{23}\omega_{34}\omega_{41}, c_2=\omega_{56}\omega_{67}\omega_{75}$ and the void cycle $1$. We get for instance,
$$\mu*f_{11}(h_1) = \mu(h_1) f_{11}(1) + \mu(c_{2})f_{11}(c_1) =  0 =  \mu_{11}(h_1),$$
using that $f_{11}(c_1) = 1$, $f_{11}(c_2)=0$ and $\mu(c_2) = -1$. From $\mu_{11}(c_1) = 0$, we also verify
 $$  \beta*\mu_{11}(h_1) = \beta(h_1)\mu_{11}(1)+ \beta(c_1)\mu_{11}(c_2)=  0 = f_{11}(h_1). $$

 \noindent \textit{Example 2.}  We now consider the graph given in Figure~\ref{fig:3}, composed of two simple cycles sharing one vertex and the covering hike $h_2 = \omega_{12}\omega_{23}\omega_{31}\omega_{24}\omega_{45}\omega_{52} $. \vspace{-0.4cm}

\begin{figure}[H]
	\centering
		\includegraphics[width=0.35\textwidth]{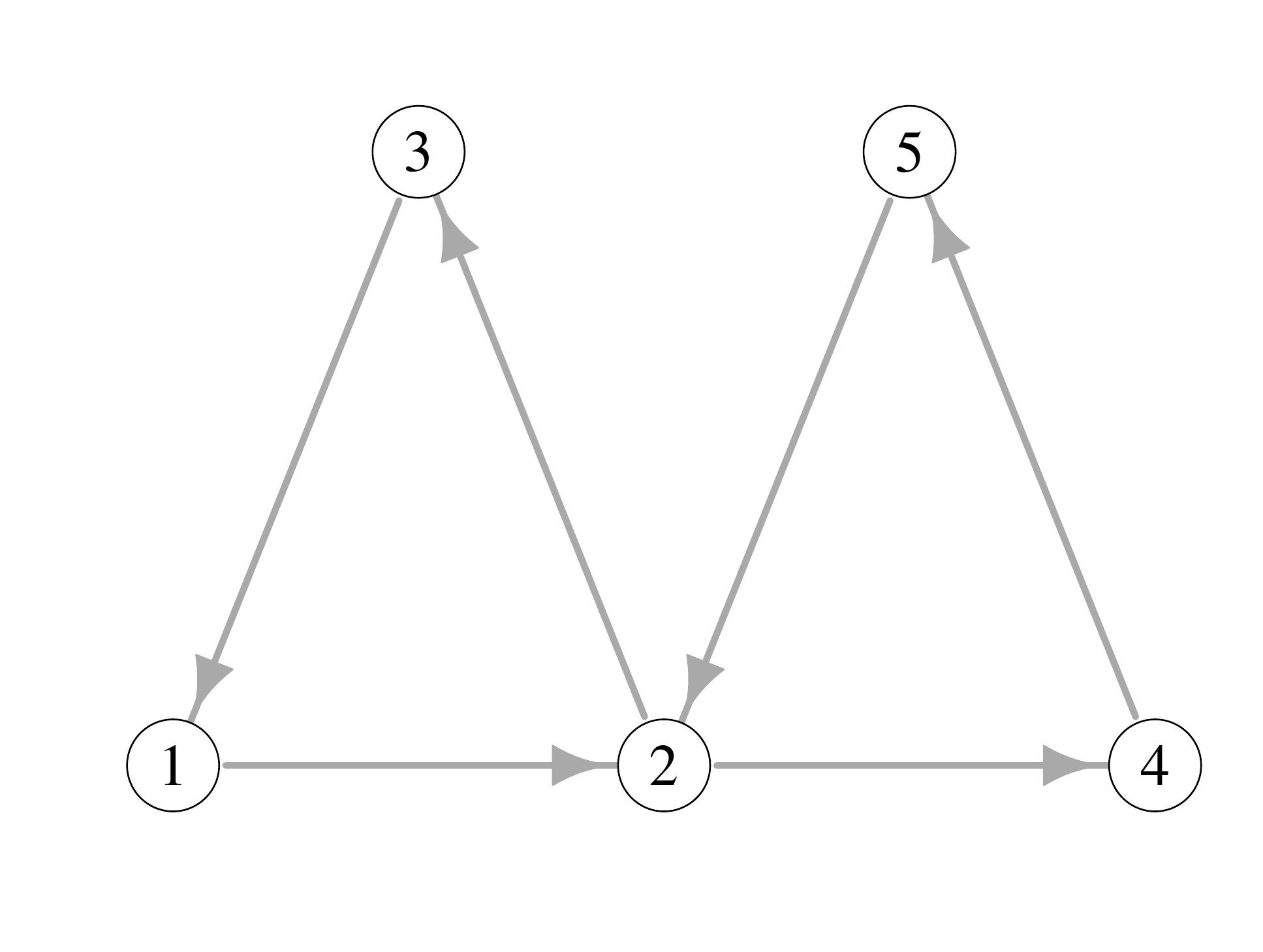}
		\vspace{-0.8cm}
  \caption{\small{Two simple cycles with one common vertex.}}
	\label{fig:3}
\end{figure}

\noindent Since $h_2$ is closed, $f_{ij}(h_2)$ is null for all $i \neq j$. Moreover, there are two ways to travel across $h_2$ starting from $v_2$, depending on which side is visited first, and one way for every other vertex. We deduce $f_{22}(h_2)=2 $ and $f_{ii}(h_2)=1$ for $i=1,3,4,5$. Since $h_2$ is not self-avoiding $\mu(h_2)=0$ and Theorem~\ref{betaexpl} gives $\beta(h_2) = 2$. To check the formulas, we now consider all the decompositions of $h_2$ into a product of closed hikes. The two non-trivial divisors of $h_2$ are $c_1=\omega_{12}\omega_{23}\omega_{31}$ and $c_2=\omega_{24}\omega_{45}\omega_{51}$. 
 %
We verify for instance,
   \begin{align*}
 \mu*f_{22}(h_2) =&~ \mu(1)f_{22}(h_2)+\mu(c_{2})f_{22}(c_1)+\mu(c_{1})f_{22}(c_2)=   0 = \mu_{22}(h_2)
 \\  \beta*\mu_{11}(h_2) =&~ \beta(h_2)\mu_{11}(1)+ \beta(c_1)\mu_{11}(c_2)=   1 = f_{11}(h_2)
 \end{align*}

\noindent \textit{Example 3.}  This example deals with the walk $h_3 = \omega_{12}\omega_{23}\omega_{34}\omega_{41}\omega_{62}\omega_{25}\omega_{54}\omega_{46}$ composed of two simple cycles sharing two vertices, represented in Figure~\ref{fig:4}. The non-trivial divisors of $h_3$ are detailed in Figure~\ref{fig:4b}.   \\

\vspace{-1.3cm}

\begin{figure}[H]
	\centering
		\includegraphics[width=0.41\textwidth]{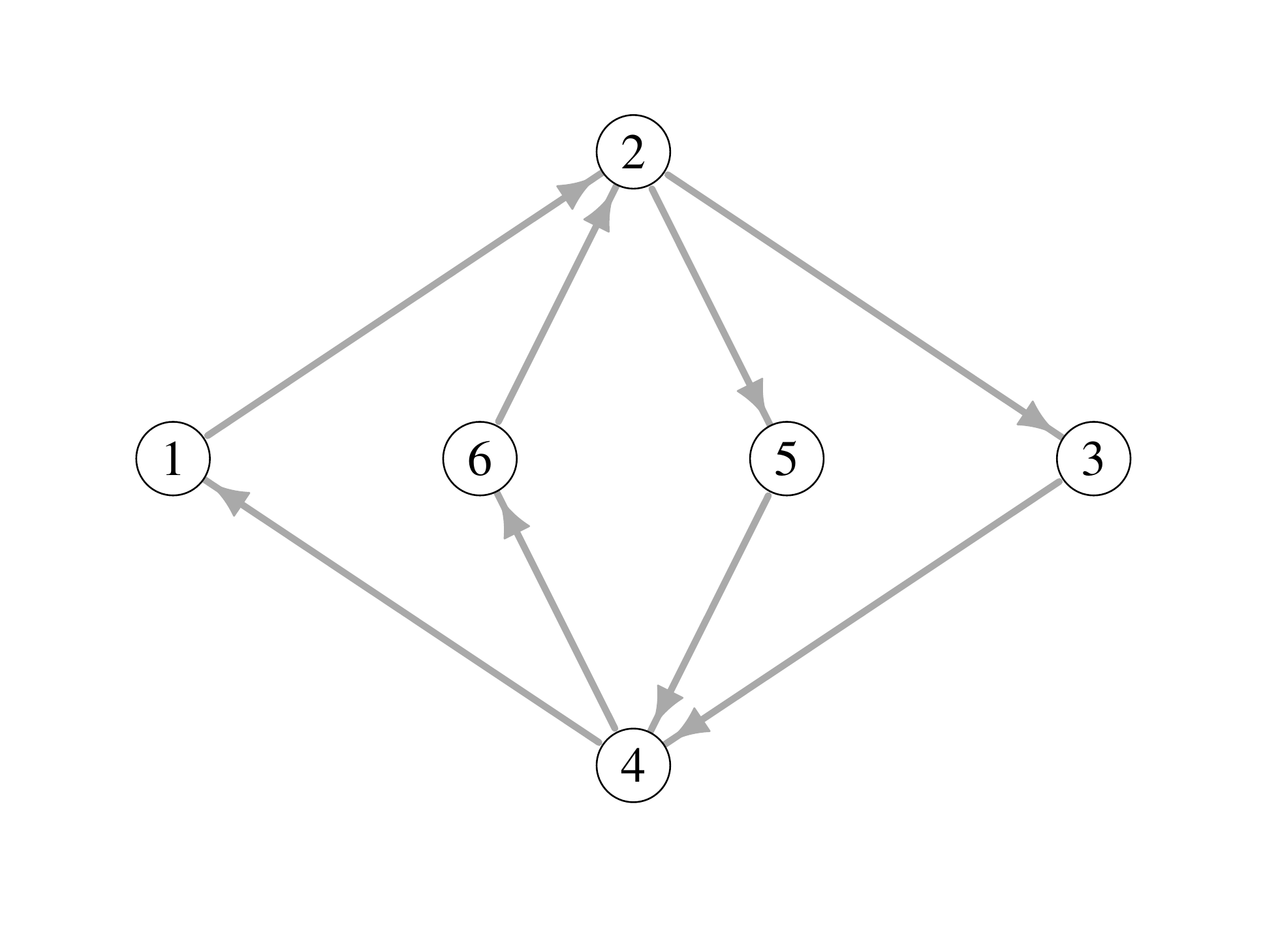}
	\vspace{-1cm}
 \caption{\small{Two simple cycles with two common vertices.}}
	\label{fig:4}
\end{figure}

\noindent Clearly, $f_{ij}(h_3) = 0$ for $i \neq j$. We compute the values $f_{ii}(h_3)$ by enumerating all ways to travel across $h_3$ starting from $v_i$. For instance, the two ways from $v_1$ are 
\begin{align*}
  & 1 \rightarrow 2 \rightarrow 3 \rightarrow  4 \rightarrow  6 \rightarrow 2 \rightarrow 5 \rightarrow 4 \rightarrow 1 \\
  & 1 \rightarrow 2 \rightarrow 5 \rightarrow  4 \rightarrow 6 \rightarrow 2 \rightarrow 3 \rightarrow 4 \rightarrow 1 
 \end{align*}
We find $f_{11}(h_3) = f_{33}(h_3)  = f_{55}(h_3) = f_{66}(h_3) = 2$ and $f_{22}(h_3) =f_{44}(h_3) = 4$. Here again, $\mu_{ij}(h_3)$ is null for all $i,j = 1, \dots ,6$ as well as $\mu(h_3)$, while $\beta(h_3) = 4$. We verify the convolution equalities $\mu_{ii}(h_3) = \mu * f_{ii}(h_3)$ and $f_{ii}(h_3) = \beta * \mu_{ii}(h_3)$ for arbitrary vertices, e.g.
   \begin{align*}
 \mu*f_{11}(h_3) =&~ \mu(1)f_{11}(h_3)+\mu(c_{2})f_{11}(c_1)+\mu(c_{3})f_{11}(c_4)=   0 = \mu_{11}(h_3)
 \\  \beta*\mu_{22}(h_3) =&~ \beta(h_3)\mu_{22}(1)=   4 = f_{22}(h_3) 
 \end{align*}

\vspace{-0.7cm}

\begin{figure}[H]
	\centering
		\includegraphics[width=0.36\textwidth]{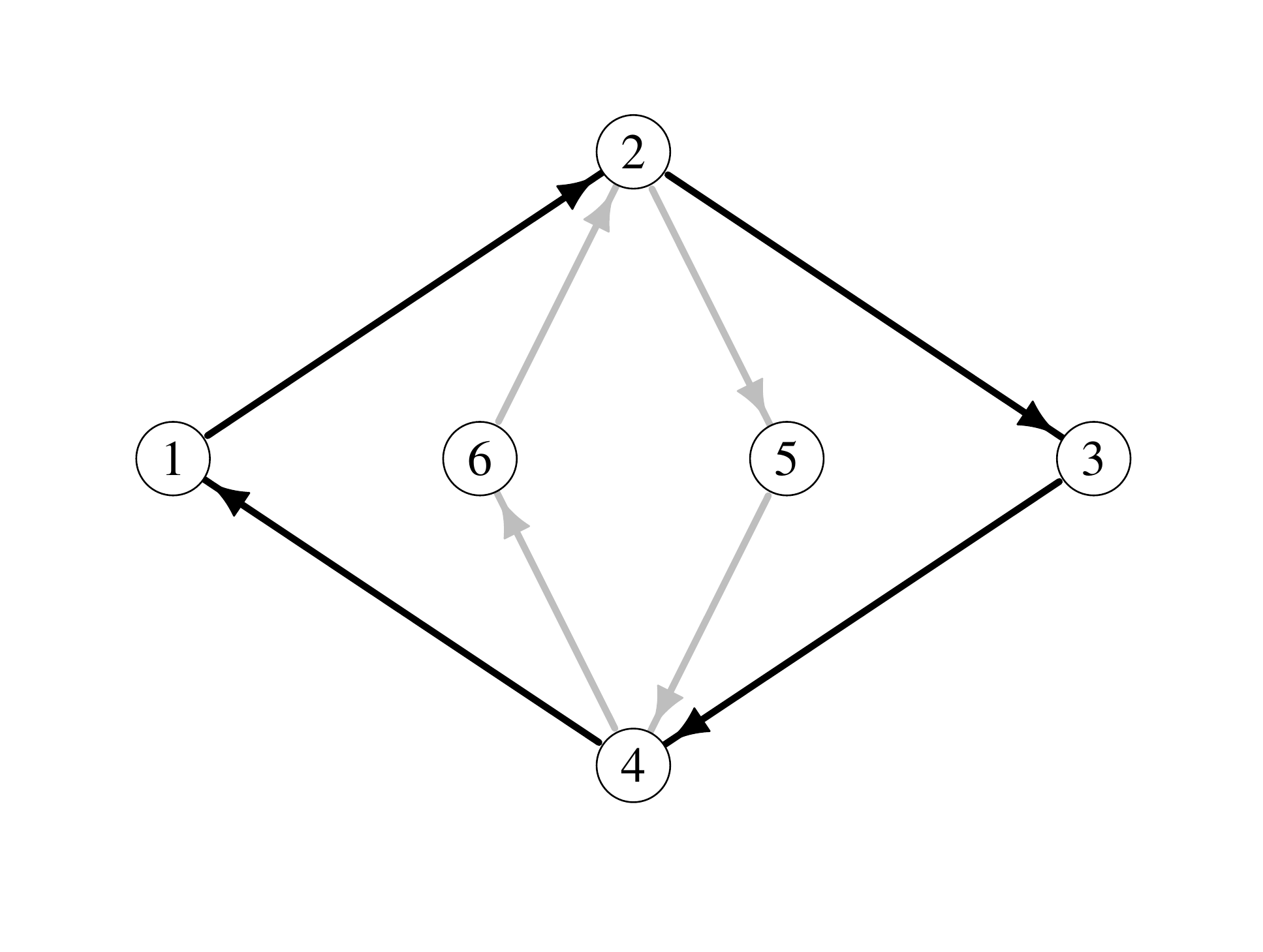} \hspace{0cm}\includegraphics[width=0.36\textwidth]{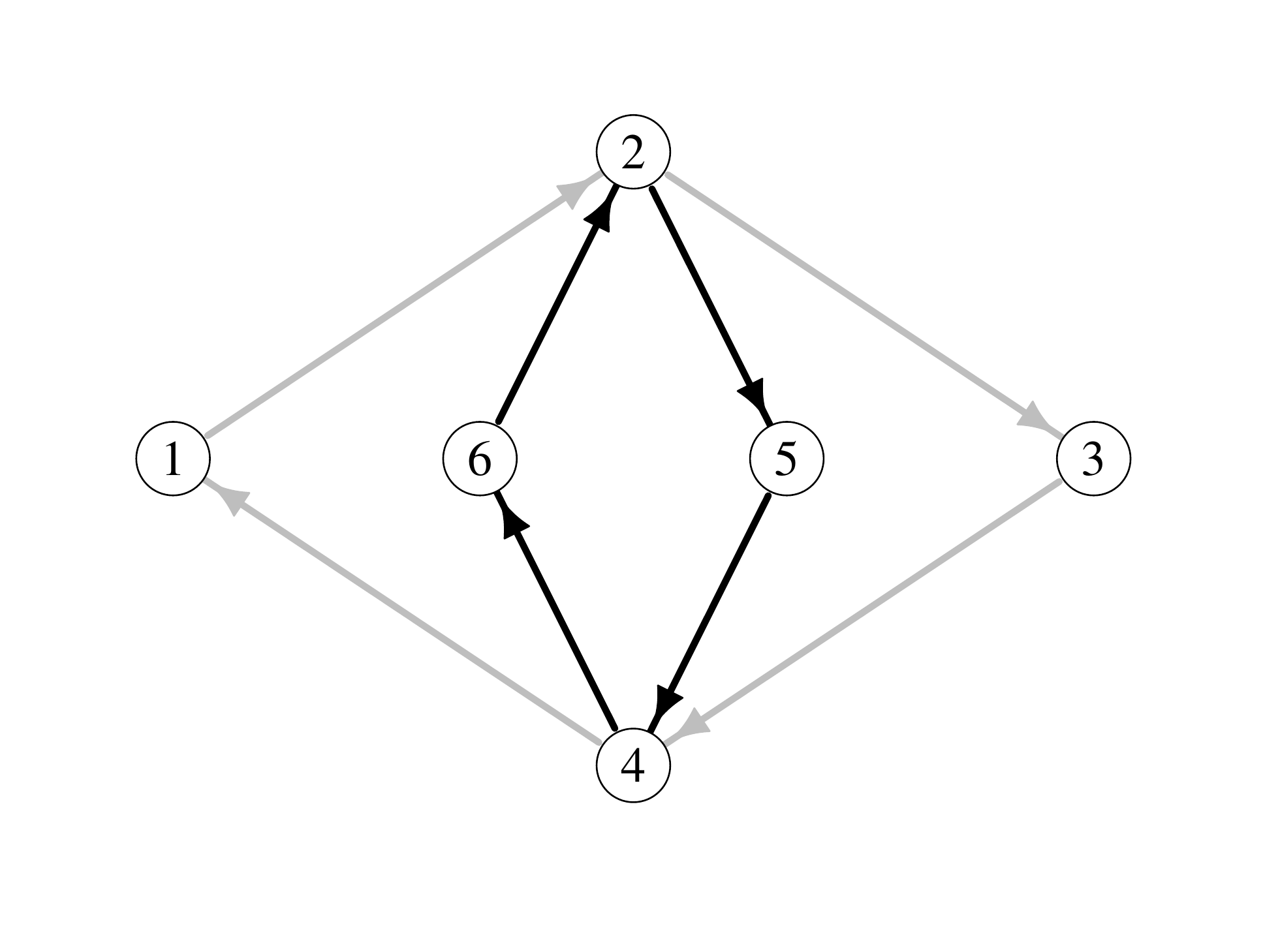}	\\
			\vspace{-0.5cm}
		\includegraphics[width=0.36\textwidth]{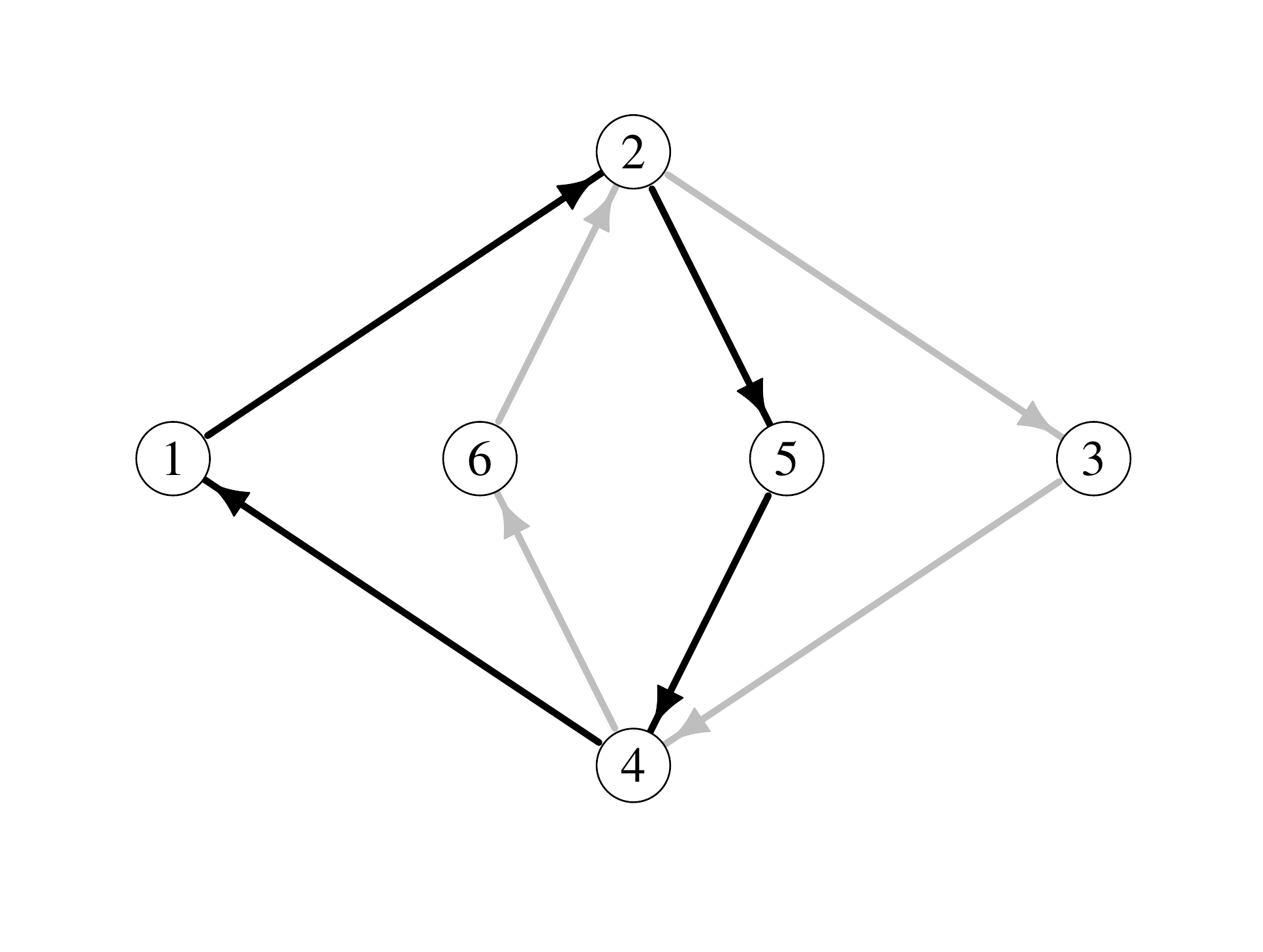} \hspace{0cm}\includegraphics[width=0.36\textwidth]{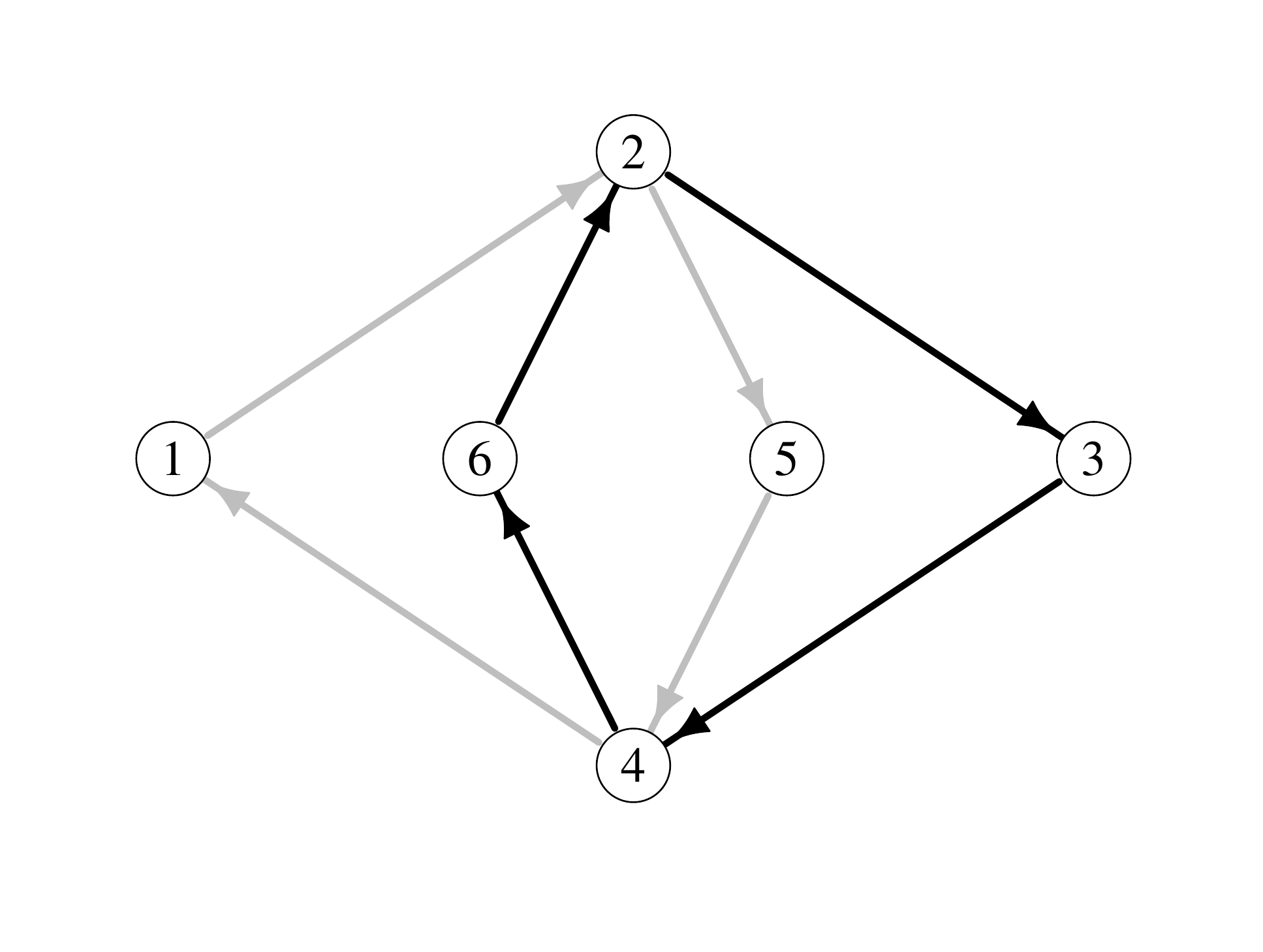}\\
		\vspace{-0.8cm}
\caption{\small{Simple divisors of $h_3$: $c_1 = \omega_{12}\omega_{23}\omega_{34}\omega_{41}$ (top-left), $c_2 = \omega_{25}\omega_{54}\omega_{46}\omega_{62}$ (top-right), $c_3 = \omega_{12}\omega_{25}\omega_{54}\omega_{41}$ (bottom-left) and $c_4 = \omega_{23}\omega_{34}\omega_{46}\omega_{62}$ (bottom-right).}} 
		\label{fig:4b}
\end{figure}

\noindent \textit{Example 4.} We consider the closed hike $h_4 =\omega_{12}\omega_{23}\omega_{35}\omega_{56}\omega_{64}\omega_{41}\omega_{25}\omega_{54}\omega_{42} $, illustrated in Figure \ref{fig:5}, composed of $2$ cycles sharing $3$ vertices. In this case, note that the orientation of the two cycles has an impact on the values of $f_{ii}$. \vspace{-0.2cm}
\begin{figure}[H]
	\centering
		\includegraphics[width=0.45\textwidth]{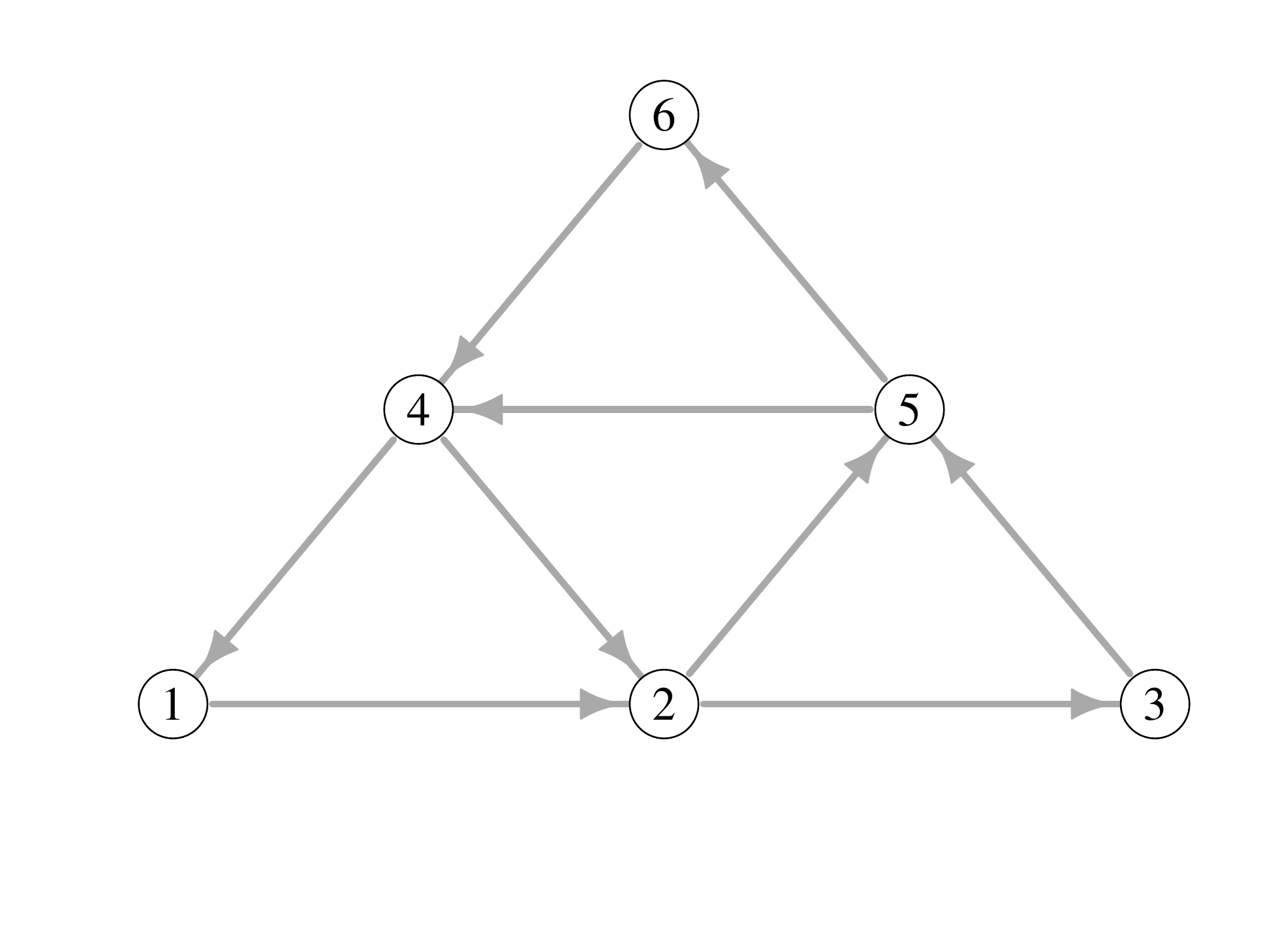}
		\vspace{-1cm}
 \caption{\small{Two simple cycles with three common vertices.}}
	\label{fig:5}
\end{figure}

\noindent  The walk $h_4$ has $8$ non-trivial divisors, detailed in Figure \ref{fig:5b}. 

\begin{figure}[H]
	\centering
		\includegraphics[width=0.35\textwidth]{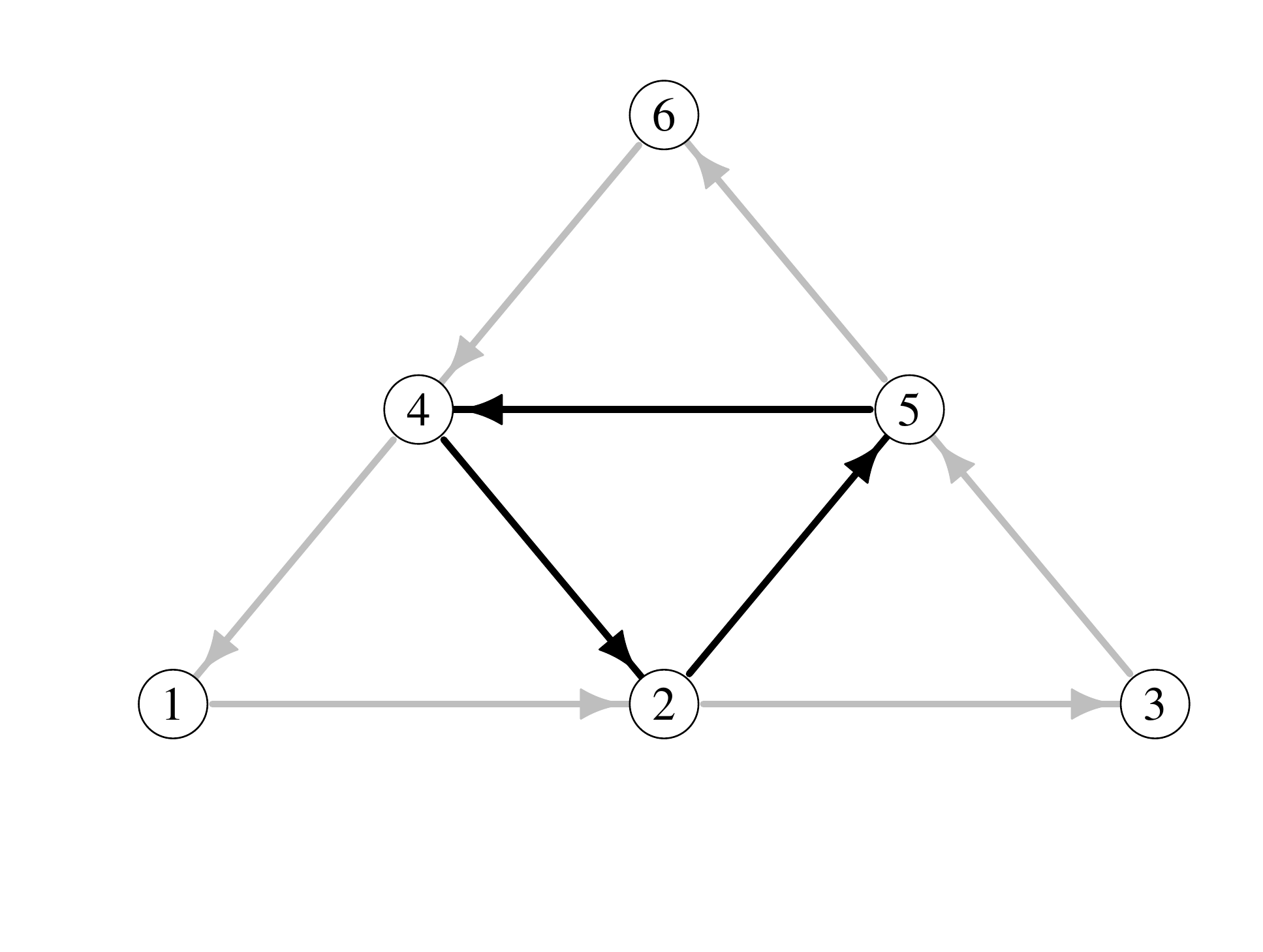} \hspace*{0.5cm} \includegraphics[width=0.35\textwidth]{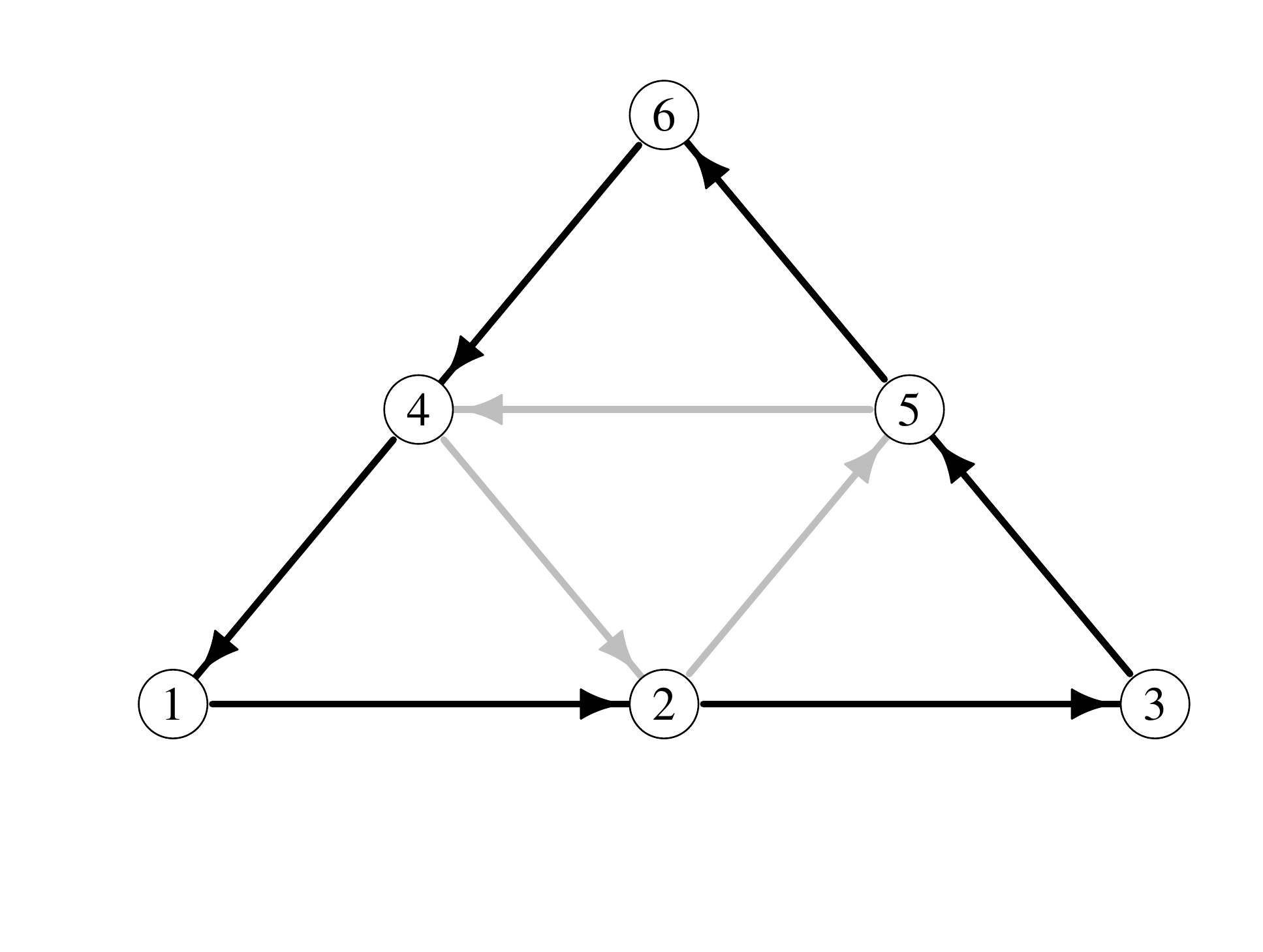} \\ \vspace{-0.8cm}
		
\includegraphics[width=0.35\textwidth]{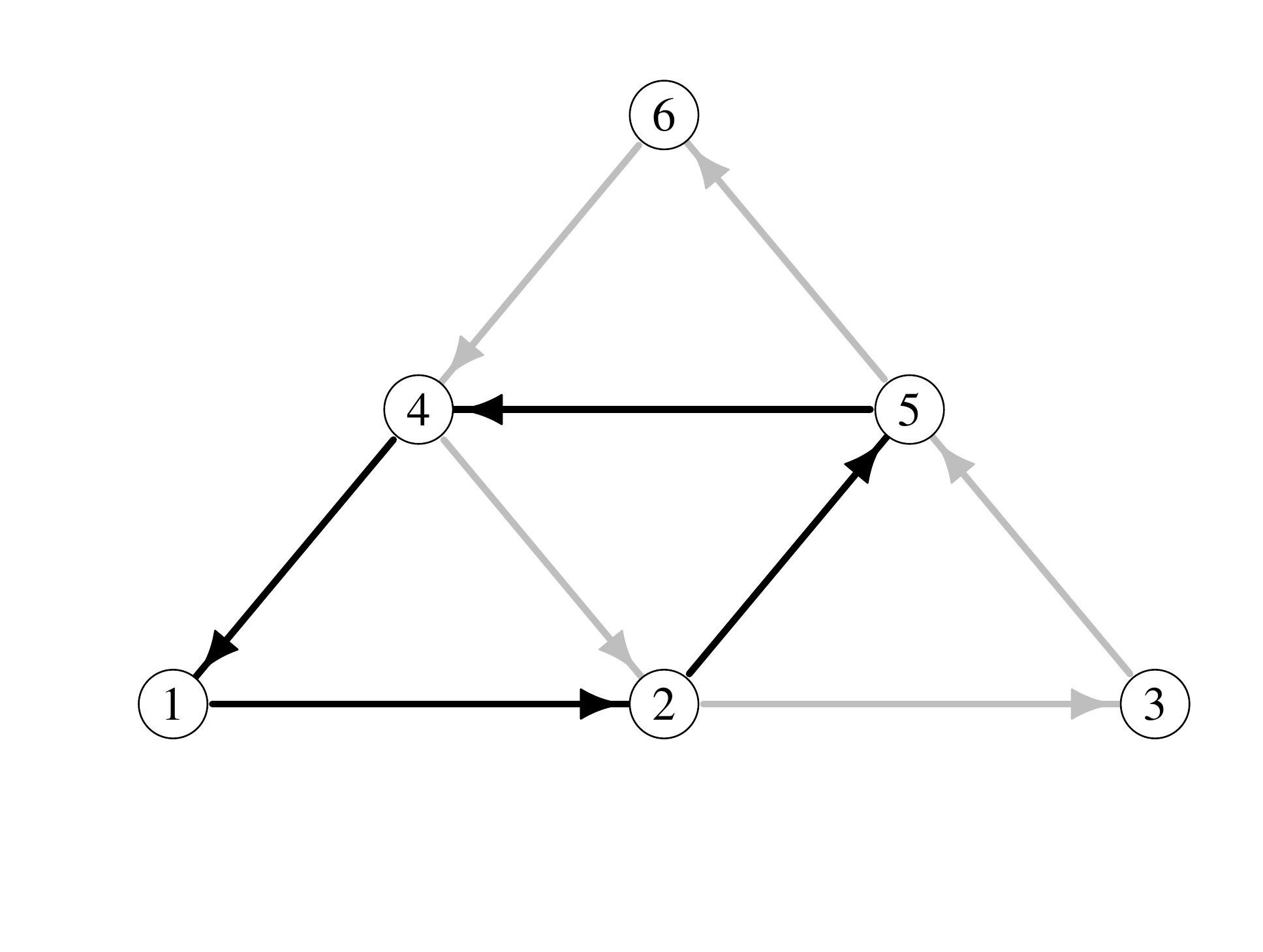} \hspace*{0.5cm} \includegraphics[width=0.35\textwidth]{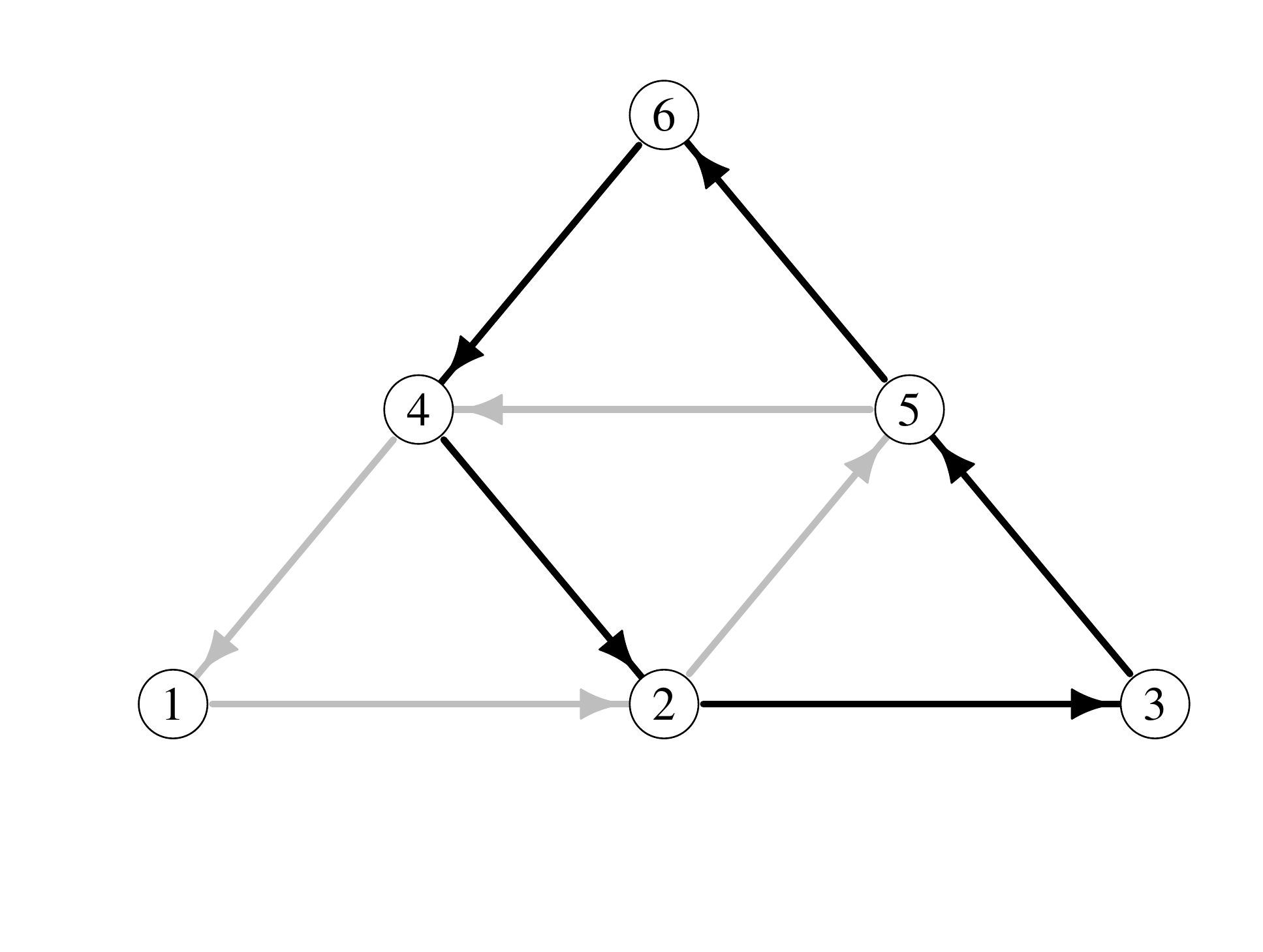} \\ \vspace{-0.8cm}

\includegraphics[width=0.35\textwidth]{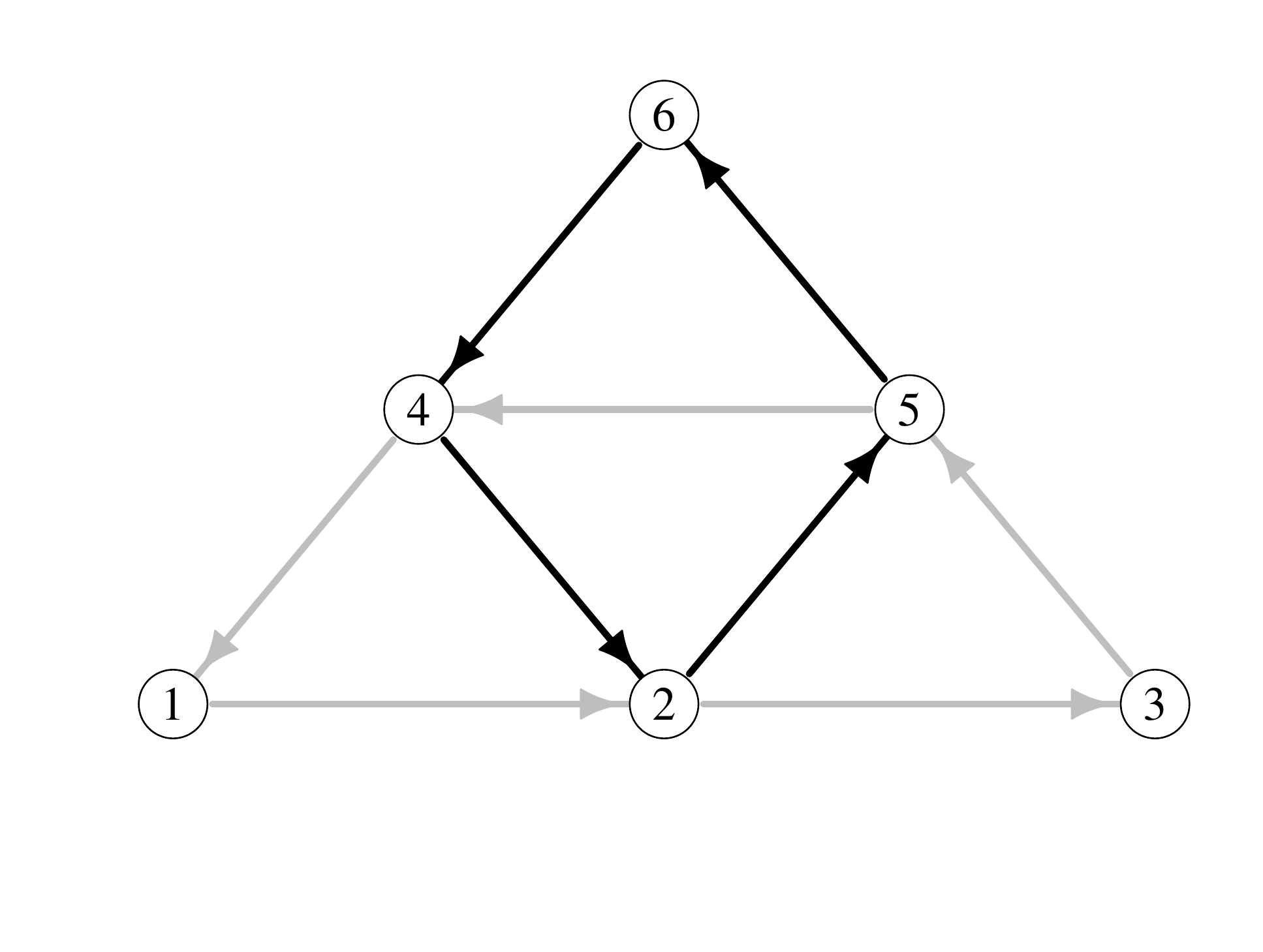} \hspace*{0.5cm} \includegraphics[width=0.35\textwidth]{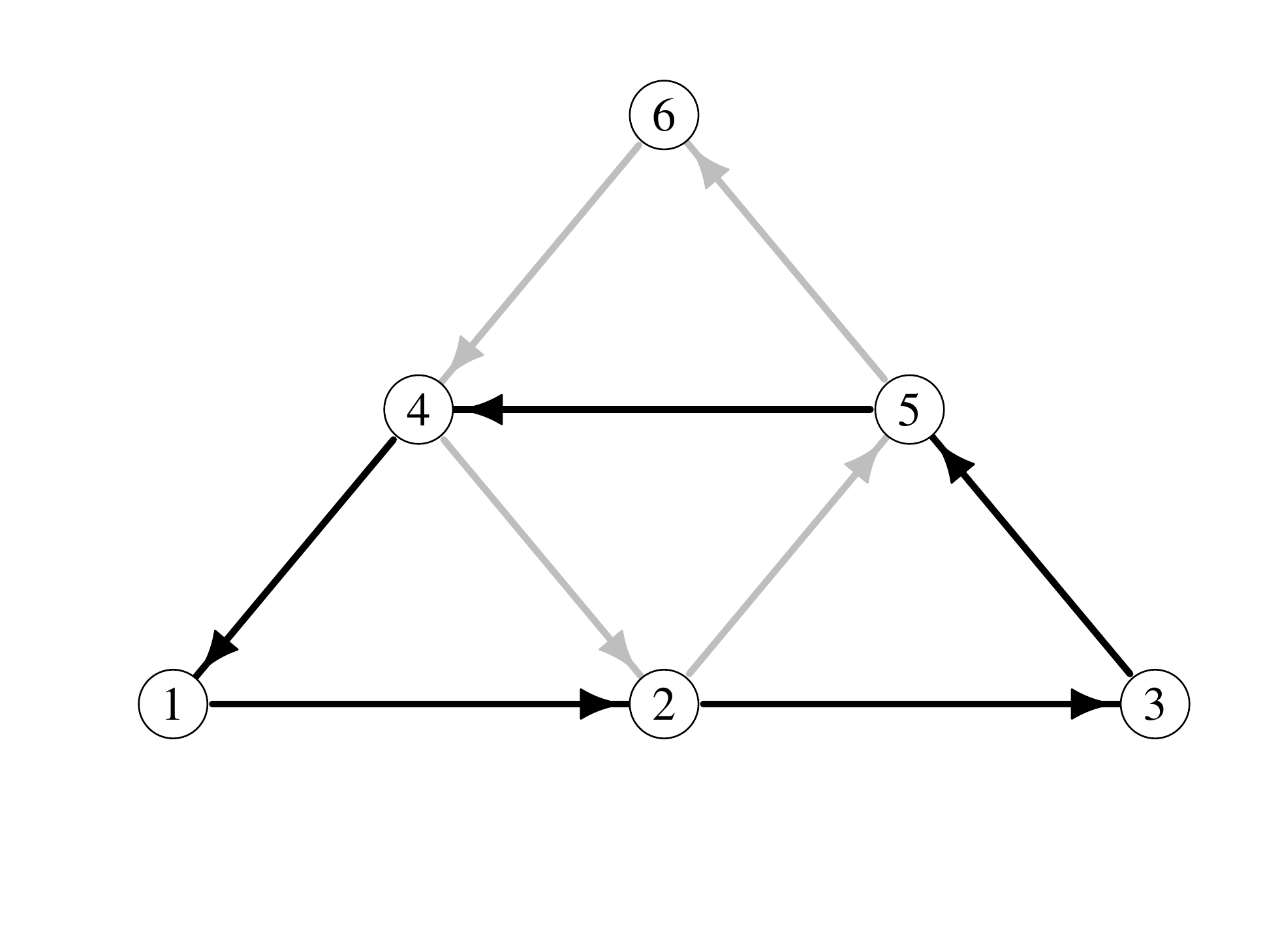} \\ \vspace{-0.8cm}

\includegraphics[width=0.35\textwidth]{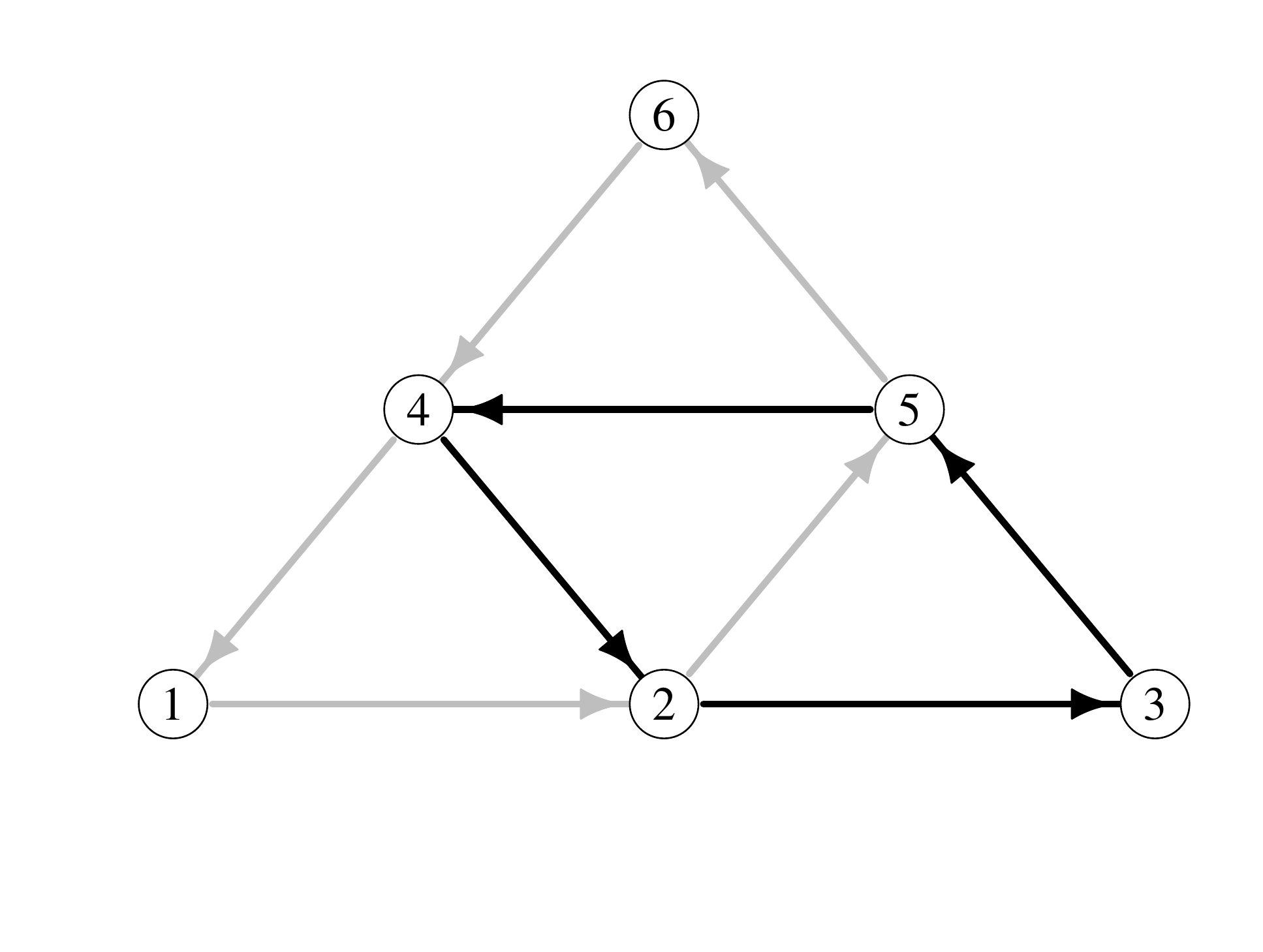} \hspace*{0.5cm} \includegraphics[width=0.4\textwidth]{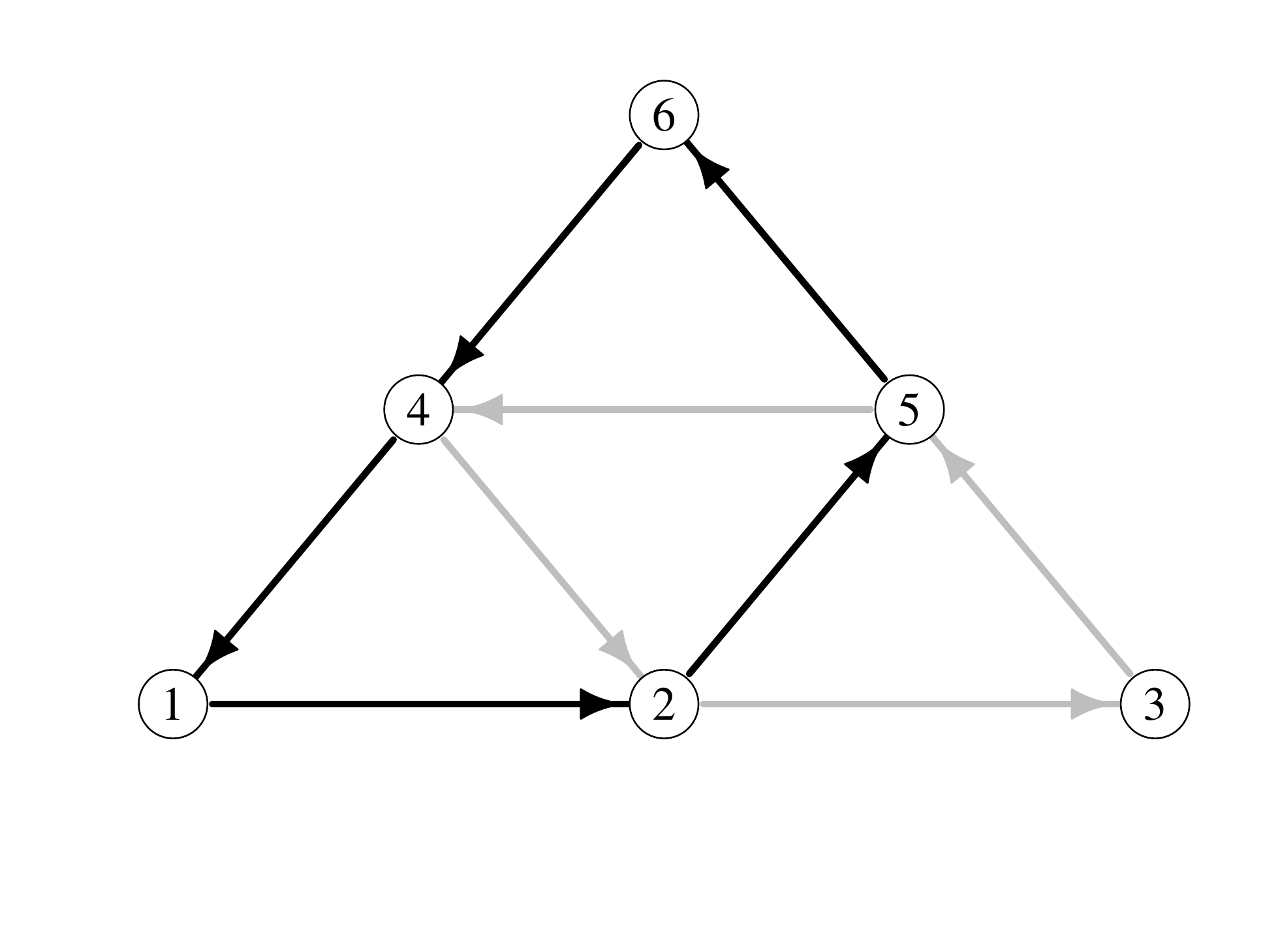} \\ \vspace{-0.6cm}
 \caption{\small{Simple divisors of $h_4$. From top-left to bottom-right: $c_1 = \omega_{25}\omega_{54}\omega_{42}$, $c_2 = \omega_{12}\omega_{23}\omega_{35}\omega_{56}\omega_{64}\omega_{41}$, $c_3 = \omega_{12}\omega_{25}\omega_{54}\omega_{41}$, $c_4 = \omega_{23}\omega_{35}\omega_{56}\omega_{64}\omega_{42}$, $c_5 = \omega_{25}\omega_{56}\omega_{64}\omega_{42}$, $c_6 = \omega_{12}\omega_{23}\omega_{35}\omega_{54}\omega_{41}$, $c_7 = \omega_{23}\omega_{35}\omega_{54}\omega_{42}$ and $c_8 = \omega_{12}\omega_{25}\omega_{56}\omega_{64}\omega_{41}$.}} 
\label{fig:5b}
\end{figure}

\noindent We know that $f_{ij}(h_4) =0$ for $i \neq j$, $\mu_{ij}(h_4)=0$ for all $i,j=1, \dots 6$, $\mu(h_4)=0$ and $\beta(h_4) = 8$. Counting the connected paths starting from each vertex gives $f_{11}(h_4)=f_{33}(h_4)=f_{66}(h_4)=4$ and $f_{22}(h_4) = f_{44}(h_4) = f_{55}(h_4) = 8$. We recover the correct values from the identities $\mu_{ii}(h_4) = \mu*f_{ii}(h_4)$ and $f_{ii}(h_4) = \beta*\mu_{ii}(h_4)$. Keeping only the non-zero values in the convolution gives, for instance
\begin{align*}
  \mu_{22}(h_4) & = \mu(1)f_{22}(h_4)+\mu(c_{2})f_{22}(c_1)+ \mu(c_{1})f_{22}(c_2)+\mu(c_{4})f_{22}(c_3)+\mu(c_{3})f_{22}(c_4) \\
	& \ \ \ \ \ + \mu(c_{6})f_{22}(c_{5})+ \mu(c_{5})f_{22}(c_6)+\mu(c_{8})f_{22}(c_7)+\mu(c_{7})f_{22}(c_8)\\ 
& = 8 - 1  - 1  - 1  - 1  - 1  - 1  - 1  - 1   =  0  \\
f_{11}(h_4) & = \beta(h_4)\mu_{11}(1)+ \beta(c_2)\mu_{11}(c_1)+\beta(c_3)\mu_{11}(c_4)+ \beta(c_6)\mu_{11}(c_5)+ \beta(c_8)\mu_{11}(c_7) \\ 
& = 8 - 1  - 1  - 1  - 1 = 4.
\end{align*}

\noindent \textit{Example 5.} Consider a closed hike $h_5$ composed of two simple cycles of opposite directions crossing $n$ times. This example can be represented as $n$ cycles placed one after the other. As we observed in the previous examples, the length of these cycles does not impact the values of the functions $f_{ii}, \mu_{ii}, \mu$ and $\beta$ so that we can take cycles of length $2$ without loss of generality, considering for instance the closed hike $\omega_{12}\omega_{21}\omega_{23}\omega_{32}\cdots \omega_{n1}\omega_{1n}$ illustrated in Figure \ref{fig:6}. 

\begin{figure}[H]
	\centering
		\includegraphics[width=0.5\textwidth]{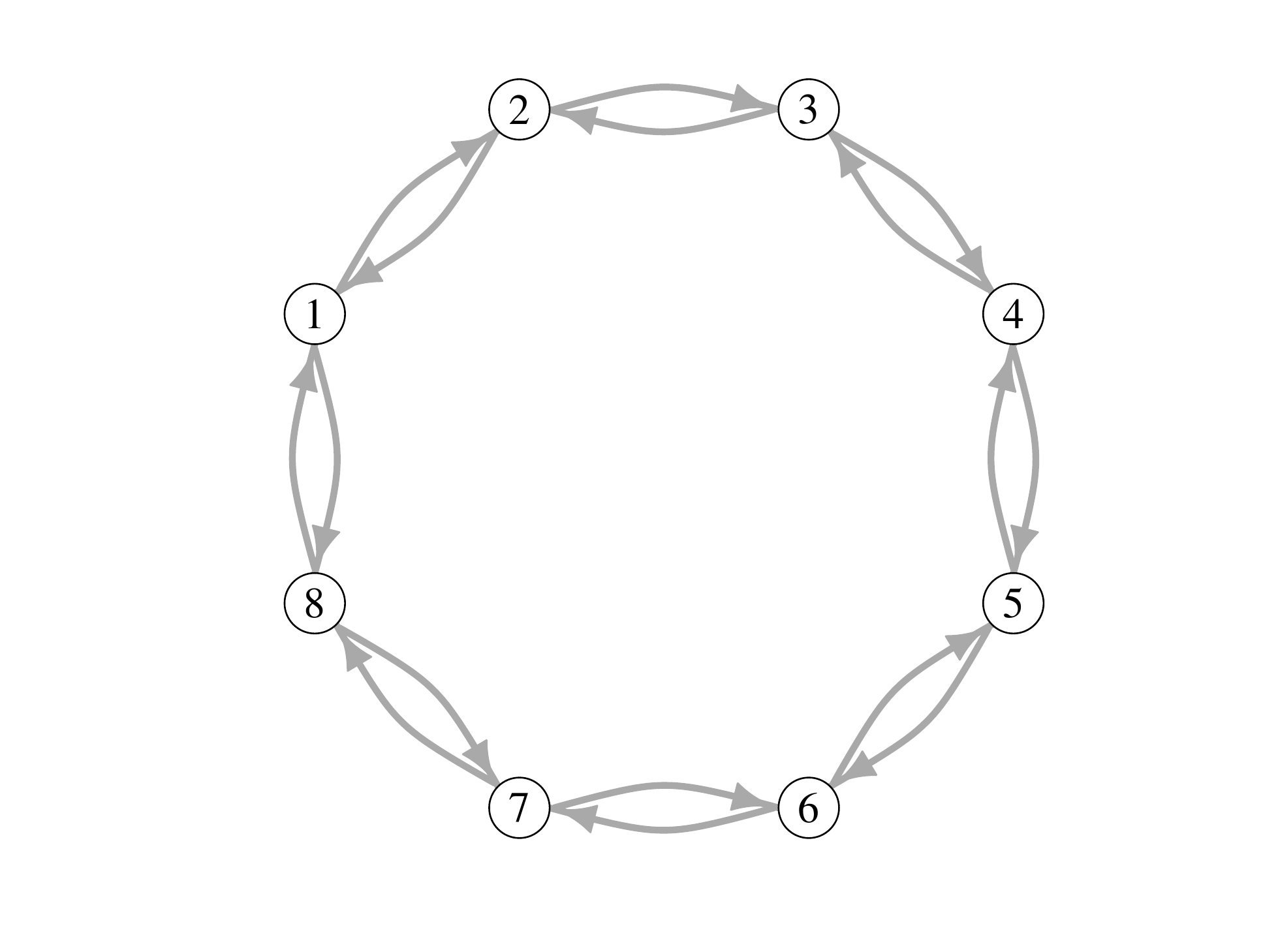}
		\vspace{-0.8cm}
 \caption{\small{Two simple cycles with $n=8$ common vertices.}} 
	\label{fig:6}
\end{figure}

\noindent We find $f_{ij}(h_5) = 0$ for $i \neq j$, $f_{ii}(h_5) = 2n$, $\mu_{ij}(h_5) = 0$ for all $i,j= 1, \dots n$, $\mu(h_5) = 0$ and $\beta(h_5)= 2^n$. Let $a_k =\omega_{k k+1}\omega_{k+1 k}, \ k =1, \dots ,n-1$, $a_n = \omega_{n1}\omega_{1n}$ be the length $2$ divisors of $h_5$ and $d_0 = \omega_{12}\omega_{23} \cdots \omega_{n-1n}\omega_{n1}$, $d_1 =\omega_{21}\omega_{32} \cdots \omega_{nn-1}\omega_{1n}$ the outside and inside cycles composing $h_5$. The non-trivial divisors of $h_5$ are $d_0, d_1$ and every product $a_{k_1}\cdots a_{k_p}$ obtained for a subset $\{k_1, \cdots , k_p\}$ of $\{1,...,n\}$. Using that $f_{11}$ is non-zero only for walks passing through $v_1$ and $\mu$ vanishes for non self-avoiding closed hikes, we obtain by keeping only the non-zero terms in the Dirichlet convolution
 \begin{align*}
\mu*f_{11}(h_5)  = \mu(1)f_{11}(h_5) + \mu (d_1)f_{11}(d_0)+ \mu(d_0)f_{11}(d_1)+ \sum_{k = 2}^{n} \mu(a_k)f_{11}\Big(\frac{h_5}{a_k}\Big) 
 \end{align*}
which recovers ultimately $ \mu*f_{11}(h_5)=  2n -2 -2(n-1) = 0 = \mu_{11}(h_5)$. The reverse relation $f_{11}(h_5)= \beta*\mu_{11}(h_5)$ is less trivial. To compute it, we have to enumerate for any $p$, the sets $\{k_1,\cdots, k_p\} \subset \{ 1,...,n \}$ such that $\mu_{11}(a_{k_1} \cdots a_{k_p}) \neq 0$, i.e., such that $a_{k_1} \cdots a_{k_p}$ is a self-avoiding closed hike that does not cross $v_1$. For each such closed hike, the complement $h_5/(a_{k_1} \cdots a_{k_p})$ is composed of $p$ disjoint connected components, one of which contains $a_1 a_n$. This component can be divided into two connected components by separating $a_1$ and $a_n$. Thus, each closed hike $a_{k_1} \cdots a_{k_p}$ such that $\mu_{11}(a_{k_1} \cdots a_{k_p}) \neq 0$ can be associated with a composition of $n-p$ containing $p+1$ elements, which there are $\binom{n-p-1}{p}$ of them. For each $a_{k_1} \cdots a_{k_p}$, the coefficient $\beta$ of the complement equals $2^{n-2p}$ and one recovers the formula
 \begin{align*}  \beta*\mu_{11}(h_5) 	= \sum_{p = 0}^{\lfloor \frac{n-1}{2}\rfloor} \binom{n-p-1}{p} (-1)^p \ 2^{n-2p} = 2n = f_{11}(h_5).
   \end{align*}

\noindent \textit{Example 6.} Consider a self-avoiding closed hike $h_6$ composed of $n\geq 2$ simple connected components $a_1,...,a_n$ (we may assume without loss of generality that each connected component is a loop as illustrated in Figure \ref{fig:7}).

\begin{figure}[H]
\hspace*{0.4cm}\includegraphics[width=0.5\textwidth]{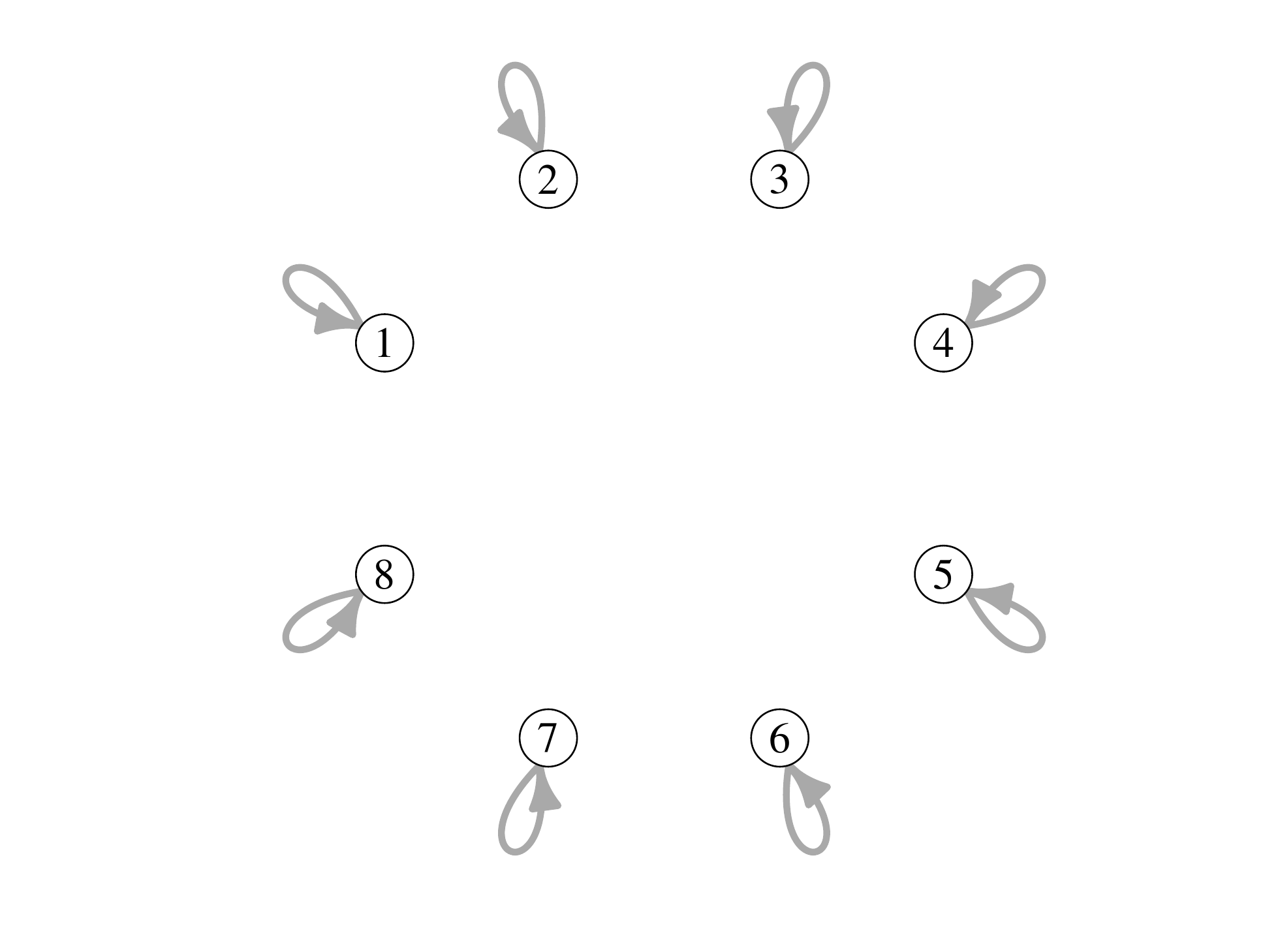}	\hspace{-1.1cm}	\includegraphics[width=0.5\textwidth]{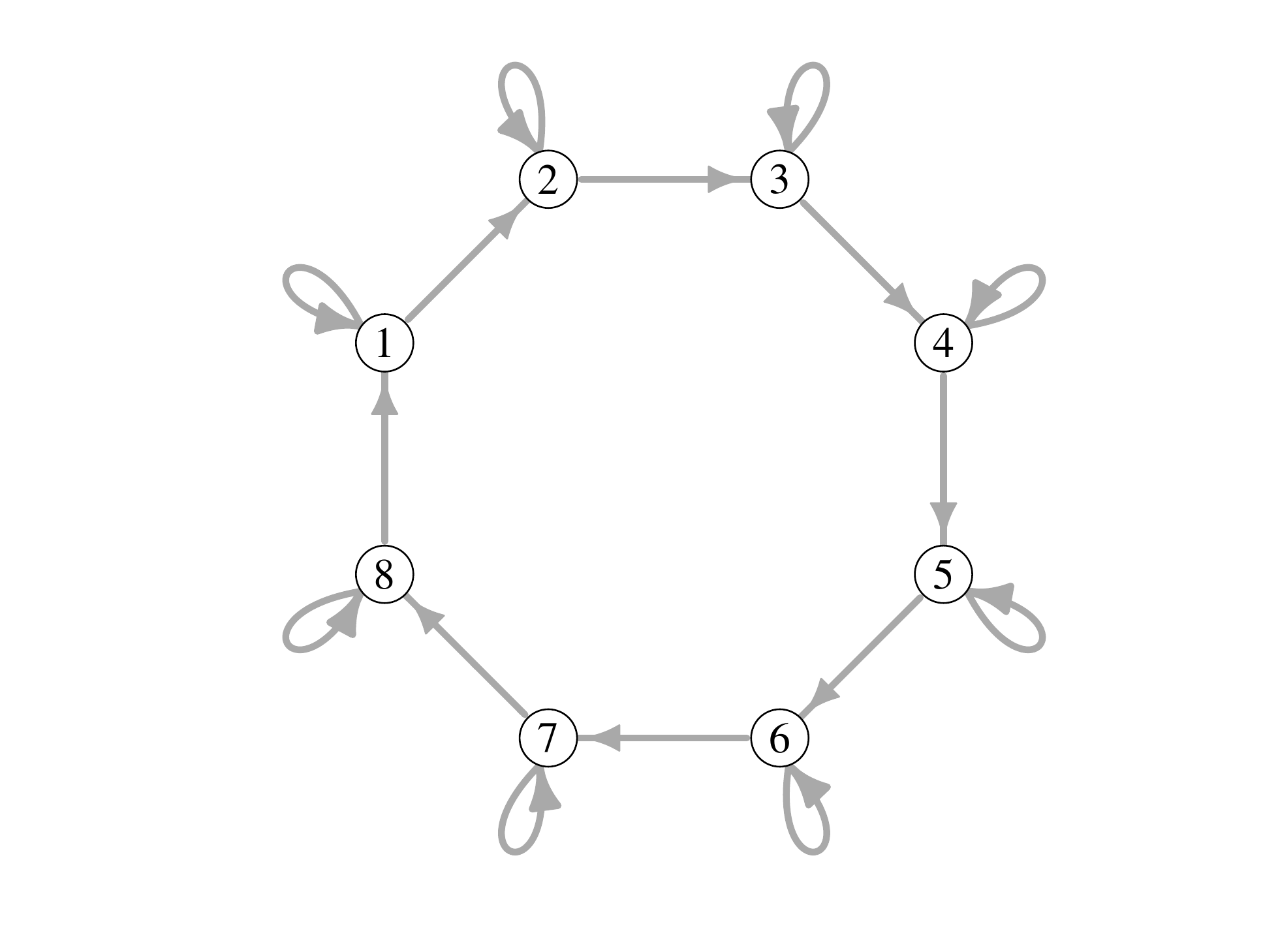}	
\vspace{-0.7cm}
	\caption{\small{Illustration of the hikes considered in Examples 6 (left) and 7 (right) for $n=8$.}}	\label{fig:7}
\end{figure}

\noindent From the first expression of $\beta$ given in Theorem \ref{betaexpl}, it is clear that $\beta(h_6) = 1$. On the other hand, the number of ways to decompose $h_6$ into a product of $k \leq n$ non-empty self-avoiding closed hikes writes as the sum of the multinomial coefficients over all positive compositions $n_1,...,n_k$ of $n$. Since for any self-avoiding decomposition $s_1,...,s_k$, the product $\mu(s_1) ... \mu(s_k)$ always equals $\mu(h_6) = (-1)^n$, combining the two expressions of $\beta(h_6)$ yields
$$ \beta(h_6) = 1 =  (-1)^n \sum_{k=1}^n \ (-1)^k \!\! \!\! \sum_{n_1 + ... + n_k = n} \frac{n!}{n_1! ... n_p!}.   $$ 
Alternatively, this equality can be obtained by identifying the coefficient of $x^k$ in the power series expansions of $e^{-x} = 1/e^x$. Since $h_6$ is not connected, we know that $f_{ii}(h_6)=0$ for all $i$. To compute the expression of $f_{ii}(h_6)$ from Theorem \ref{walk-covering}, we consider the self-avoiding decompositions of the form $h_6=s_1...s_k p$ with $s_1,...,s_k \in \mathcal C \cap \mathcal S \setminus \{ 1 \}$ and $p \in \mathcal S_{ii}$. To verify that this expression gives $f_{ii}(h_6)=0$ in this case simply observe that any self-avoiding decomposition $s_1, ..., s_k, p$ such that $p \neq \omega_{ii}$ cancels out with the decomposition $s_1, ..., s_k, p/\omega_{ii}, \omega_{ii}$ in view of
$$ (-1)^{n'(s_1)+...+n'(s_k)+n'(w)} = - (-1)^{n'(s_1)+...+n'(s_)+n'(p/\omega_{ii}) + n'(\omega_{ii})}.  $$
Thus, summing over all self-avoiding decompositions recovers $f_{ii}(h_6)=0$. \\

\noindent \textit{Example 7.} We now consider the closed hike $h_7$ constructed from the previous example with an extra cycle $c_0$ passing through each vertex: $h_7 = h_6 \times c_0$ (e.g. in the right graph in Figure \ref{fig:7} where $c_0=\omega_{12}\omega_{23}\omega_{34}\omega_{45}\omega_{56}\omega_{67}\omega_{78}\omega_{81}$). In this example, the cycle $c_0$ is isolated in every self-avoiding decomposition since it shares a common node with all the other divisors of $h_7$. Thus, the different ways to express $h_7$ as a product of self-avoiding closed hikes can be obtained from the previous example, inserting the cycle $c_0$ wherever possible. Precisely, for a decomposition $h_6=s_1...s_k$ of $h_6$ into $k \leq n$ non-empty self-avoiding closed hikes, there are exactly $k+1$ possibilities to insert $c_0$. Moreover, remark that $\mu(s_1)...\mu(s_k) \mu(c_0) = \mu(h_6) \mu(c_0) = (-1)^{n+1}$ is constant over all self-avoiding decompositions. Combining the two expressions of $\beta(h_7)$ thus recovers the formula
	$$ \sum_{k =1}^n (-1)^{k+1} (k+1) \!\!\!\!  \sum_{n_1 + ... + n_k = n}\!\!\!\!  (-1)^{n+1} \frac{n!}{n_1! ... n_k!} = 2^n.  $$
In this example, there are two ways of visiting the whole walk $h_7$ from one vertex $v_i$ to itself, depending on whether the loop at $v_i$ is traveled at the start or at the end. Thus, $f_{ii}(h_7) = 2$ for all $v_i$. In a self-avoiding decomposition with $s_1,...,s_k \in \mathcal C \cap \mathcal S\setminus \{ 1 \}$ and $p \in \mathcal S_{ii}$ we can distinguish the cases $p = c_0$, $p=\omega_{ii}$ and $p \neq \omega_{ii}, c_0$. Clearly, the sum over all self-avoiding decompositions $h_7= s_1...,s_k p$ such that $p= c_0$ yields $\beta(h_6)$ since $(-1)^{n'(c_0)} = 1$. Moreover, the sum over all self-avoiding decompositions with $p = \omega_{ii}$ recovers $\beta(h_7/\omega_{ii}) = 2^{n-1}$. Finally, for a self-avoiding decomposition $h_6 =s_1 ... s_k p$ of $h_6$ with $p \neq \omega_{ii}$, there are $k$ possibilities to insert $c_0$, yielding
$$ f_{ii}(h_7) = \beta(h_6) + \beta \Big( \frac {h_7} {\omega_{ii}} \Big)  + \sum_{k =1}^{n-1} (-1)^{k+1} k \!\!\!\!\!\! \sum_{n_1 + ... + n_k = n-1} \!\!\!\!\!\! (-1)^{n+1} \frac{(n-1)!}{n_1! ... n_k!} = 1 + 2^{n-1} - 2^{n-1} + 1 = 2. $$

\section*{Acknowledgments}
The authors are grateful to Pierre-Louis Giscard for his explanations on the poset structure of hikes, and to an anonymous referee for its helpful comments which helped improve this paper.

 \bibliographystyle{plain}
 \bibliography{biblio}  

\end{document}